\documentclass[11pt]{amsart}
\usepackage{hyperref}
\usepackage{amsmath}
\usepackage{amssymb}
\usepackage{amsthm}
\usepackage{mathrsfs}
\usepackage{latexsym}
\usepackage{amsthm}
\usepackage{graphicx}
\usepackage{color}
\usepackage[all]{xy}
\usepackage{cite}
\usepackage{booktabs}

\usepackage[full]{textcomp}
\usepackage[sups]{Baskervaldx} 
\usepackage{cabin} 
\usepackage[varqu,varl]{inconsolata} 
\usepackage[baskervaldx,bigdelims,vvarbb]{newtxmath} 
\usepackage[cal=cm]{mathalfa} 
\usepackage{tikz}
\usetikzlibrary{arrows,matrix,decorations.pathmorphing,cd}

\theoremstyle{plain}
\newtheorem{theorem}{Theorem}[section]
\newtheorem{corollary}[theorem]{Corollary}
\newtheorem{proposition}[theorem]{Proposition}
\newtheorem{lemma}[theorem]{Lemma}

\theoremstyle{remark}
\newtheorem{remark}[theorem]{Remark}
\theoremstyle{definition}
\newtheorem{definition}[theorem]{Definition}
\newtheorem{construction}[theorem]{Construction}

\newtheorem{proposition-definition}[theorem]{Proposition-Definition}

\newtheorem{notation}[theorem]{Notation}

\DeclareMathOperator{\codim}{codim}
\DeclareMathOperator{\Cok}{Cok}
\DeclareMathOperator{\Def}{Def}
\DeclareMathOperator{\ev}{ev}    
\DeclareMathOperator{\GW}{GW}
\DeclareMathOperator{\vdim}{vdim}
\DeclareMathOperator{\ctop}{c_{\text{\rm top}}}

\newcommand{\limiting}{\text{\rm lim}}
\newcommand{\red}{\text{\rm red}}
\newcommand{\vv}{\text{\rm virt}}

\newcommand{\A}{\mathbb{A}}
\newcommand{\E}{\mathbb{E}}
\newcommand{\pp}{\mathbb{P}}
\newcommand{\Q}{\mathbb{Q}}
\newcommand{\Z}{\mathbb{Z}}

\newcommand{\frA}{\mathfrak{A}}
\newcommand{\frC}{\mathfrak{C}}
\newcommand{\frD}{\mathfrak{D}}
\newcommand{\frE}{\mathfrak{E}}
\newcommand{\frF}{\mathfrak{F}}
\newcommand{\frG}{\mathfrak{G}}
\newcommand{\frL}{\mathfrak{L}}
\newcommand{\frM}{\mathfrak{M}}
\newcommand{\frN}{\mathfrak{N}}
\newcommand{\frP}{\mathfrak{P}}
\newcommand{\frX}{\mathfrak{X}}

\DeclareMathOperator{\frPic}{\mathfrak{Pic}}
\newcommand{\Picsms}{\frPic^{\text{st}}}
\newcommand{\tildefrPic}{\tilde{\frPic}}

\newcommand{\tildeMP}{\tilde{M}(\pp)}
\newcommand{\tildeMX}{\tilde{M}(X)}
\newcommand{\tildeMXP}{\tilde{M}^X(\pp)}
\newcommand{\tildeV}{\tilde{V}}
\newcommand{\tildeW}{\tilde{W}}
\newcommand{\tildeA}{\tilde{A}}
\newcommand{\tildefrM}{\tilde{\frM}}

\newcommand{\frBlF}{\mathfrak{Bl(F)}}
\newcommand{\frBlG}{\mathfrak{Bl(G)}}
\newcommand{\BltildeV}{Bl(\tildeV)}
\newcommand{\BltildeMP}{Bl(\tildeMP)}
\newcommand{\BltildeMX}{Bl(\tildeMX)}
\newcommand{\BltildeMXP}{Bl(\tildeMXP)}
\newcommand{\BlF}{Bl(F)}
\newcommand{\BlGzero}{Bl(G^0)}
\newcommand{\BlGlambda}{Bl(G^{\lambda})}

\DeclareMathOperator{\DD}{DD}
\DeclareMathOperator{\Corr}{Corr}

\title[A Splitting of the Virtual Class]{A splitting of the virtual class for \\ genus one stable maps}
\author{Tom Coates}
\email{t.coates@imperial.ac.uk}
\author{Cristina Manolache}
\email{c.manolache@imperial.ac.uk}
\address{Department of Mathematics \\ Imperial College London\\ 180 Queen's Gate \\ London SW7 2AZ \\ United Kingdom}

\begin{document}
\maketitle
\begin{abstract} Moduli spaces of stable maps to a smooth projective variety typically have several components. We express the virtual class of the moduli space of genus one stable maps to a smooth projective variety as a sum of virtual classes of the components. The key ingredient is a generalised functoriality result for virtual classes. We show that the natural maps from `ghost' components of the genus one moduli space to moduli spaces of genus zero stable maps satisfy the strong push forward property. As a consequence, we give a cycle-level formula which relates standard and reduced genus one  Gromov--Witten invariants of a smooth projective Calabi--Yau theefold.\end{abstract}
\tableofcontents

\section{Introduction}
 Let $X$ be a smooth projective variety in $\pp^r$. Let $\bar{M}_{g,n}(X,d)$ be the moduli space of stable maps to $X$ with genus $g$ and homology class $d \in H_2(X;\Z)$~\cite{k}. Then $\bar{M}_{g,n}(X,d)$ has virtual dimension 
\[
\vdim = n + (1-g)(\dim X - 3) + c_1(TX) \cdot d
\]
and a virtual class $[\bar{M}_{g,n}(X,d)]^{\vv}\in A_{\vdim}(\bar{M}_{g,n}(X,d))$: see~\cite{lt1, lt2, bf}. Gromov--Witten (GW) invariants of $X$ are intersection numbers against this virtual class.  They are related to counts of curves in $X$ of genus $g$ and class $d$. 

 For $g> 0$, GW invariants also encode some contributions from degenerate lower genus stable maps. These contributions are fairly well understood for genus one GW invariants of threefolds. In genus one, Zinger and Li--Zinger prove a formula which expresses GW invariants in terms of \emph{reduced} invariants (which are closely related to BPS numbers) and degenerate contributions. From now on, we restrict ourselves to the genus-one case. 

 The degenerate contributions reflect the structure of the moduli space of stable maps, which has many components, some of which contribute to GW invariants. For example, the moduli space of stable maps $M(\pp)=\bar{M}_{1,n}(\pp^r,d)$ has a main component $M(\pp)^0$, whose generic point is a map from a smooth genus one curve, and 
 several other components 
\[
M(\pp)^{\lambda}\simeq M_{1,k+n_0}(\pp^r,0)\times_{\pp^r}\bar{M}_{0,1+n_1}(\pp^r,d_1)\times_{\pp^r}\ldots\times_{\pp^r}\bar{M}_{0,1+n_k}(\pp^r,d_k) \, \big/ \, \Gamma^\lambda.
\]
Here $\lambda$ denotes the combinatorial data $(k; n_0,\ldots,n_k; d_1,\ldots, d_k)$ and $\Gamma^\lambda$ is the (finite) automorphism group of this data. With this notation we have $M(\pp)=M(\pp)^0\cup\bigcup_{\lambda\in I} M(\pp)^{\lambda}$ for an appropriate index set~$I$.
 
 The first step in the analysis is the definition of reduced invariants~\cite{z,VZ}. The idea is to construct a blow-up $\tildeMP$ of the moduli space of stable maps, which induces a blow up of $M(X)$. Consider 
\[
\tildeMX=\tildeMP\times_{M(\pp)}M(X)=\tildeMX^0\cup \bigcup_{\lambda \in I} \tildeMX^{\lambda}.
\]
On $\tildeMX^0$ it is possible to define a virtual class \cite{z,VZ}.  Reduced GW invariants are intersection numbers against this virtual class.

Following \cite{lichang,huli}, we will refer to $\tildeMP^\lambda$ and $\tildeMX^\lambda$, $\lambda \in I$, as `ghost components'.  In this paper, we define virtual classes on the ghost components $\tildeMX^{\lambda}$ and prove:

 \begin{theorem} \label{main}  We have the following equality in $A_*(\tildeMX)$:
 \begin{equation} \label{eq:main}
   [\tildeMX]^{\vv}= [\tildeMX^0]^{\vv}+\sum_{\lambda \in I}[\tildeMX^{\lambda}]^{\vv}.
 \end{equation}
 \end{theorem}

 \noindent There is a natural projection from the boundary component $\tildeMX^{\lambda}$ to the space 
\[
P(X)^{\lambda} = \bar{M}_{0,1+n_1}(\pp^r,d_1)\times_{\pp^r}\ldots\times_{\pp^r}\bar{M}_{0,1+n_k}(\pp^r,d_k) \, \big/ \, \Gamma_\lambda,
\]
which forgets the collapsed genus-$1$ component.  $P(X)^{\lambda}$ carries a natural virtual class $[P(X)^{\lambda}]^{\vv}$, and we prove:

 \begin{theorem} \label{main2}Let $X$ be a smooth projective threefold. Then, the morphisms $$q^{\lambda}:\tildeMX^{\lambda}(X)\to P(X)^{\lambda}$$ satisfy the strong virtual push-forward property.
 \end{theorem}

 \noindent Theorems~\ref{main} and~\ref{main2} together give a cycle-level proof of the Zinger/ Li--Zinger formula~\cite{z,lz2,zingerstvsred}. 

 \begin{theorem}\label{thm:LZ_formula} Let $X$ be a Calabi--Yau threefold. Then, the reduced invariants and GW invariants of $X$ are related by the formula 
   \[
     \GW_{1,\beta}^X=\GW^{X,\red}_{1,\beta}+\frac{1}{12}\GW^{X}_{0,\beta}.
   \]
 \end{theorem}
\noindent Theorem~\ref{thm:LZ_formula} is a particular case of \cite{z} which holds for $X$ any compact symplectic manifold of dimension 2 and 3 and of \cite{zingerstvsred} which holds for $X$ any compact symplectic manifold of any dimension. A similar statement appears in \cite{LZ, lz2}. Algebraically it has been proved by Chang and Li~\cite{lichang}, for $X$ the quintic threefold, using a slightly different definition for reduced Gromov--Witten invariants to Zinger \cite{z}. See below for a detailed discussion.The algebraic method uses the moduli space of maps with fields~\cite{lichang-fields}. Note that unlike \cite{lz2,lichang} we do not require $X$ to be a complete intersection. 

 \subsection*{Outline of the proof} 

The key ingredient in our splitting of the virtual class \eqref{eq:main} is a functoriality property for virtual classes.  This generalises the functoriality result of Kim--Kresch--Pantev~\cite{kkp,eu}.  Their setting is the following.  Suppose that we are given 
\begin{itemize}
\item DM-type morphisms of stacks $i \colon F\to G$ and $j \colon G\to H$;
\item a compatible triple ($E_{F/G}^\vee$, $E_{F/H}^\vee$, $E_{G/H}^\vee$) of perfect dual obstruction theories, that is, perfect dual obstruction theories $E_{F/G}^\vee$, $E_{F/H}^\vee$, $E_{G/H}^\vee$ which sit in a commutative diagram
  \[
    \xymatrix{
      i^* E_{G/H} \ar[r] \ar[d] & E_{F/H} \ar[r] \ar[d] & E_{F/G} \ar[r]^{+1} \ar[d] & \\
      i^* L_{G/H} \ar[r] & L_{F/H} \ar[r] & L_{F/G} \ar[r]^{+1} & \\
    }
  \]
  where the rows are distinguished triangles.  Let us write 
  \begin{align*}
    \frE_{F/G}  = h^1/h^0(E_{F/G}^\vee), && 
    \frE_{F/H}  = h^1/h^0(E_{F/H}^\vee), &&
    \frE_{G/H}  = h^1/h^0(E_{G/H}^\vee)
  \end{align*}
  for the vector bundle stacks determined by the obstruction theories.
\end{itemize}
These data determine
\begin{enumerate}
\item a morphism $F\times\pp^1\to \Def _{G} H$; and
\item a perfect dual obstruction theory $E_{F\times\pp^1/\Def _{G} H}$;
\end{enumerate}
where  $\Def _{G} H$ is the deformation to the normal cone~\cite{f,Kresch}.  In turn, these determine a family of cone stacks~\cite{bf}, and an embedding of this family into a vector bundle stack.  On the general fiber, this is
\[
C_{F/H}\to \frE_{F/H}
\]
and the special fibre 
\[
C_{F/H}^{\limiting} \to \frE_{F/G}\oplus i^* \frE_{G/H}.
\]
satisfies 
\[
[C_{F/H}^{\limiting}] = [C_{F/C_{G/H}}].
\]
The latter equality holds in the Chow group of the double deformation space $\Def_{F \times \pp^1} \Def_{G} H$.  It follows that
\[
i^! [G]^\vv = [F]^\vv.
\]

We would like to apply this with $F = M(X)$, $G = M(\pp)$, and $H$ equal to the Picard stack $\frPic$, and then argue that, since $M(\pp)$ has components, this gives a splitting of the virtual class: 
\[
[M(X)]^\vv = i^![M(\pp)^0] + \sum_{\lambda \in I} i^![M(\pp)^\lambda]^\vv.
\]
The problem with this is that $E_{M(X)/M(\pp)}$ is not perfect, and so in particular $i^!$ does not make sense. Also, we do not have a perfect dual obstruction theory $E_{M(X)\times\pp^1/\Def _{M(\pp)} \frPic}$. However, inspired by cosection localisation~\cite{KLcosection}, we will resolve these problems by blowing up.

The first step is to replace $M(X)$, $M(\pp)$, and $\frPic$ by the Vakil--Zinger blow-ups $\tildeMX$, $\tildeMP$, and $\tildefrPic$.  The dual obstruction theory $E_{\tildeMX/\tildeMP}$ is still not perfect but it is now a union of vector bundles, which have different ranks on the different components $\tildeMX^\lambda$, $\lambda \in {0} \cup I$.  Focus now on a ghost component $\tildeMX^\lambda$.  The restriction to $\tildeMX^\lambda$ of $E_{\tildeMP/\tildefrPic}$ is too large to sit in a compatible triple (of restrictions)
\[
\left(\text{$E_{\tildeMX/\tildeMP}^\vee$, $E_{\tildeMX/\tildefrPic}^\vee$, $E_{\tildeMP/\tildefrPic}^\vee$}\right).
\]
so we would like to construct a reduced version of $E_{\tildeMP/\tildefrPic}$ on $\tildeMX^\lambda$.  To do this, we take further blow-ups $\BltildeMP$ of $\tildeMP$ and $\BltildeMX$ of $\tildeMX$.  After restricting to the ghost component $\BltildeMX^\lambda$ of $\BltildeMX$, there is a compatible triple
\[
\left(\text{$E_{\BltildeMX/\BltildeMP}^\vee$, $E_{\BltildeMX/\tildefrPic}^\vee$, $E_{\lambda,\red}^\vee$}\right)
\]
of dual obstruction theories, where $E_{\lambda,\red}^\vee$ is a complex defined on $\BltildeMX^\lambda$ which plays the role of a reduced dual obstruction theory.  The complex $E_{\lambda,\red}^\vee$ is perfect, and allows us to define virtual classes on the ghost components\footnote{There is a small lie here.  In the main text we define the virtual classes of ghost components after passing to the further blow-up $\Def^\prime_{\BltildeMP} \tildefrPic$.  But we could have used $E_{\lambda,\red}^\vee$ instead.}.  It also solves the second problem mentioned above, at least on the ghost components: $E_{\lambda,\red}^\vee$ induces a perfect dual obstruction theory  
\begin{equation}
  \label{induced on lambda}
  E_{\BltildeMX\times\pp^1/\Def _{\BltildeMP} \tildefrPic}\big|_{\BltildeMX^\lambda}.
\end{equation}

It remains to consider the main component $\BltildeMX^0$.  In this case $E_{\BltildeMP/\tildefrPic}^\vee$ has the correct rank, but it fails to sit in a distinguished triangle
\[
\left(\text{$E_{\BltildeMX/\BltildeMP}^\vee$, $E_{\BltildeMX/\tildefrPic}^\vee$, $E_{\BltildeMP/\tildefrPic}^\vee$}\right)
\]
and the induced obstruction theory
\begin{equation}
  \label{induced on main}
  E_{\BltildeMX\times\pp^1/\Def_{\BltildeMP} \tildefrPic}\big|_{\BltildeMX^0}
\end{equation}
fails to be perfect along a divisor $\delta$ in the special fiber $\BltildeMX\times\{0\}$.  These are essentially the same problem.  We resolve it by blowing up the deformation space $\Def_{\BltildeMP} \tildefrPic$ along $\delta$, obtaining a new space $\Def^\prime_{\BltildeMP} \tildefrPic$, and then truncating the pullback of the obstruction theory \eqref{induced on main}.  By construction this truncation is perfect; it has the same general fiber as before, but a different special fiber.

Write $Z(X)^0$ for the blow-up of $\BltildeMX^0\times\pp^1$ and $Z(X)^{\lambda}$ for the blow-up of $\BltildeMX^0\times\pp^1$.  At this point we have vector bundle stacks   
\[
\mathfrak{H}^0 \to Z_0
\]
 built from the obstruction theory~(\ref{induced on main})
and 
\[
\mathfrak{H}^{\lambda} \to Z_{\lambda}.
\]
 built from the obstruction theory~(\ref{induced on lambda}). These vector bundle stacks contain families of cones, with general fibres 
\begin{equation*}
\begin{array}{rcl}
C_{\BltildeMX/\tildefrPic}^0&&\text{supported on $\BltildeMX^0$ and}\\
C_{\BltildeMX/\tildefrPic}^\lambda&&\text{supported on $\BltildeMX^{\lambda}$.}
\end{array}
\end{equation*}
 The class of special fibre supported on the main component can be written as 
\begin{equation*}[C_{\BltildeMX^0/C^0_{\BltildeMP\tildefrPic}}]+ [\text{correction class}] 
\end{equation*}
in $A_*(\mathfrak{H}^0)$ and the class of the special fibre supported on the ghost components can be written as 
\begin{equation*}
[C_{\BltildeMX^{\lambda}/C^{\lambda}_{\BltildeMP\tildefrPic}}]+ [\text{correction class}]
\end{equation*}
in $A_*(\mathfrak{H}^{\lambda})$. Here $C^0_{\BltildeMP\tildefrPic}$ and $C^{\lambda}_{\BltildeMP\tildefrPic}$ denote the components of $C_{\BltildeMP\tildefrPic}$ supported on the main component and on the ghost component respectively. We show that when we sum over $\lambda$ in $\{0\} \cup I$, the correction classes cancel. This is done in the proof of Theorem~\ref{main theorem}. 

The classes 
\begin{align*}[C_{\BltildeMX^0/C^0_{\BltildeMP\tildefrPic}}] \in A_*(\mathfrak{H}^0)&&\text{and}&&[C_{\BltildeMX^{\lambda}/C^{\lambda}_{\BltildeMP\tildefrPic}}]\in A_*(\mathfrak{H}^{\lambda})
\end{align*} 
define virtual classes that satisfy:
\[ 
  [\BltildeMX]^{\vv}=[\BltildeMX^0]^{\vv} + \sum_{\lambda \in I} [\BltildeMX^{\lambda}]^{\vv} 
\]
Pushing this splitting forward to $\tildeMX$ proves Theorem~1. 
The final step is to show that the splitting behaves well with respect to push forward. This is the content of section \S\ref{contribution}.

Theorem~\ref{main} is a particular case of a functoriality property for virtual classes in the presence of a $3$-term obstruction theory. We expect that a more general functoriality statement would solve many related questions. For example, this would apply to derive an analogue of Theorem~\ref{thm:LZ_formula} in higher genus. We will address these topics in future work.  

\subsection*{Relation to other works} Reduced genus 1 invariants are the output of a long and impressive project. Reduced invariants were defined, using symplectic methods, and compared to Gromov--Witten invariants by Zinger\cite{zsharp,zstructure, z,zingerstvsred}. Li--Zinger showed \cite{LZ, lz2} that reduced Gromov--Witten invariants are the integral of the top Chern class of a sheaf over the main component of $M(\pp)$; this is an analog, for reduced genus 1 invariants, of the quantum Lefschetz hyperplane property \cite{LZ, lz2}. In view of \cite{z} this also gives a proof of Theorem~\ref{thm:LZ_formula}. The algebraic definition requires a blow-up construction for the moduli space of stable maps to projective space due to Vakil and Zinger \cite{VZ,VZpreview}. Explicit local equations for this blow-up are given in \cite{zsharp,huli}. A modular interpretation of reduced invariants via log maps has been given by Ranganathan, Santos-Parker and Wise \cite{rspw}. More recently, reduced invariants for the quintic threefold have been compared to Gromov--Witten invariants using algebro-geometric methods by Chang and Li \cite{lichang}. As we do, Chang--Li \emph{define} reduced invariants as the integral against the top Chern class class of a sheaf but, as discussed above, this gives the same reduced invariants as \cite{z}. The algebraic comparison relies on the construction of maps with fields due to Chang and Li \cite{lichang-fields}, and on Kiem--Li's cosection localised virtual class \cite{KLcosection}. A new proof of this comparison for complete intersections in projective spaces appears in \cite{leeoh}. Zinger has computed reduced invariants of projective hypersurfaces via localisation~\cite{zingred}. See~\cite{zsurvey} for a survey from the symplectic perspective.

Reduced Gromov--Witten invariants are also related to Gopakumar--Vafa invariants \cite{gv1,gv2}, and they coincide with Gopakumar--Vafa invariants for Fano targets \cite{degcont}. Indeed the Gopakumar--Vafa invariants are by definition related to Gromov--Witten invariants by a recursive formula which takes into account degenerate lower genus and lower degree boundary contributions. These contributions were computed by Pandharipande in \cite {degcont}. Recently, reduced invariants have been related to invariants from maps with cusps \cite{bcm2}.

\newpage

\section{Notation}

The following is a table of the most frequently used notations in the paper.

\begin{table}[h!]
  \centering
  \begin{tabular}{ll}
    \toprule
    $X$ & a smooth projective variety \\ 
    
    $N$ & the normal bundle $N_{X/\pp^r}$ of $X$ in $\pp^r$\\

    $\frM_{1,n}$ or $\frM$ & the moduli space of prestable genus-one curves with \\
        & $n$ marked points\\

    $\tildefrM$ & the Vakil--Zinger blow-up of $\frM_{1,n}$\\

    $\frPic$  & the relative Picard stack over $\frM_{1,n}$ parametrizing line \\
    & bundles of degree~$\beta$ \\
    
    $\tildefrPic$ & the relative Picard stack over $\tildefrM_{1,n}$ \\

    $\Xi_i$ & exceptional divisors on $\tildefrPic$ \\
    
    $M(\pp)$ & the moduli space of genus one stable maps to $\pp^r$ \\
    
    $C(\pp)$ & the universal curve over $M(\pp)$\\
    
    $f:C(\pp)\to \pp$ & the universal stable map\\
    
    $L$ & the universal line bundle on $C(\pp)$\\
    
    $\tildeMP^{\lambda}$ & a ghost component of $\tildeMP$ \\
    
    $M_0(\pp)$ & the moduli space of genus zero maps to $\pp^r$\\
    
    $M_0(X)$ & the moduli space of genus zero maps to $X$\\

    $M(X)$ & the moduli space of genus 1 stable maps to $X$\\

    $\tildeMX$ & the product $M(X)\times_{\frM}\tildefrM$ \\

    $\tildeMX^0$ & the main component of $\tildeMX$ \\

    $\tildeMX^{\lambda}$ & a ghost component of $\tildeMX$\\

    $A$ & a divisor on a family of curves which is given by a section \\
    & of the family which meets every genus one subcurve\\

    $\sigma$ & a section of $\mathcal O(A)$ \\

    $V$ & a smooth space over $\frPic$ in which we embed $M(\pp)$ \\
    
    $\tildeV$ & a smooth space over $\tildefrPic$ in which we embed $\tildeMP$ \\
    
    $N^0$ & the main component of $\pi_*\ev^*N$\\

    $N^{\lambda}$ & the component of $\pi_*\ev^*N$ supported on $\tildeMP^{\lambda}$\\

    $\BltildeV$ & a further blow up of $\tildeV$ \\
    
    $\BltildeMP$ & the fibre product $\tildeMP\times _{\tildeV}\BltildeV$\\
    
    $\BltildeMX$ & the fibre product $\tildeMX\times _{\tildeV}\BltildeV$\\
    
        $C_{F/G}$ & the normal cone (stack) of a DM type morphism $F\to G$ \\ 
    & of algebraic stacks\\    

    $E^{\bullet}_{F/G}$ & an obstruction theory for a DM type morphism $F\to G$ \\
    & of algebraic stacks\\
    
    $\Def_FG$ & the deformation space of $G$ to $C_{F/G}$\\ \bottomrule \\
  \end{tabular}
  \caption{Frequently used notation}
\end{table}

\newpage

\section{Reduced Gromov--Witten Invariants}

In this section we define the genus-one reduced Gromov--Witten invariants of a smooth projective variety $X$, following Vakil--Zinger~\cite{VZ, VZpreview}.  This is an algebro-geometric version of the symplectic story developed in~\cite{z}.  We begin by explaining how to embed the moduli space $M(\pp)$ of genus-one stable maps to projective space into a smooth stack, giving a construction due to Ciocan-Fontanine--Kim~\cite{cfk}.  We then introduce the Vakil--Zinger blow-up of the moduli space $M(\pp)$, and use it to define the reduced invariants of $X$.

\subsection{Embeddings of $\bar{M}_{1,n}(\pp^r,\beta)$ after Ciocan-Fontanine--Kim}\label{embsmooth}\label{cfk}  We will now describe an embedding of $M(\pp)$ into a smooth stack $V$, following~\cite{cfk}.  Let $\frM_{1,n}$ denote the stack of prestable genus-one curves with $n$ marked points, and let $\frPic \to \frM_{1,n}$ denote the relative Picard stack of line bundles of degree~$\beta$.  It is well known that $\frM_{1,n}$ is a smooth Artin stack of dimension $n$. The stack $\frPic$ is smooth over $\frM_{1,n}$ of relative dimension zero, by~\cite[Remark~2.5]{eu3}; thus $\frPic$ is also smooth of dimension $n$.  Let $\pi \colon \frC \to \frPic$ be the universal curve and $\frL \to \frC$ be the tautological line bundle.  Let $\Picsms$ denote the open substack of $\frPic$ obtained by imposing the stability condition
\[
\text{$\omega_\pi(\frP_1+\cdots+\frP_n) \otimes \frL^3$ is $\pi$-relatively ample}
\]
where $\frP_1,\ldots,\frP_n$ are the divisors in $\frC$ defined by the marked points.   Slightly abusing notation, we will denote by $\pi \colon \frC\to \Picsms$ the universal curve, and by $\frL \to \frC$ the tautological line bundle.  

We now construct a smooth Deligne--Mumford stack $V$ into which $M(\pp)$ will embed.  The stack $M(\pp)$ parametrizes, up to isomorphism, tuples
\begin{equation*} 
\big(C;p_1,...,p_n;L; u_0,\ldots,u_r\big) 
\end{equation*} 
where 
\begin{enumerate}
  \renewcommand{\labelenumi}{(\roman{enumi})}
\item $C$ is a nodal curve of arithmetic genus one;
\item $p_1,\ldots,p_n$ are distinct marked smooth points on $C$;
\item $L$ is a line bundle on $C$ of degree $\beta \cdot H$;
\item $u_0, \ldots,u_r$ are global sections of $L$;
\end{enumerate}
such that $\omega_C(p_1+\cdots+p_n)\otimes L^3$ is ample and that the base locus of $u_0,\ldots,u_r$ is empty.  Let $\pi \colon C(\pp) \to M(\pp)$ denote the universal family.  Choose a $\pi$-relatively very ample effective Cartier divisor on the universal family $\frC \to \frPic$.  Let $\frA \to \frC$ denote the corresponding line bundle and $\sigma \colon \frC \to \frA$ denote the corresponding section.  The map $M(\pp) \to \frPic$ that classifies $L$ induces a map $C(\pp) \to \frC$, and we pull back $\frA$ along this map.  This gives a line bundle $\tildeA \to C(\pp)$ which is very ample on each fiber of $C(\pp) \to M(\pp)$; we have that $R^1 \pi_* (L \otimes \tildeA\,)$ vanishes on $M(\pp)$.

Consider now the total space $\frX$ of the bundle $\pi_* \big(\frL \otimes \frA\big)^{\oplus(r+1)}$ over $\Picsms$. This is an Artin stack which parametrizes, up to isomorphism, tuples
\begin{equation*} \big(C;p_1,...,p_n;L; v_0,\ldots,v_r\big) \end{equation*} 
where \begin{enumerate}
  \renewcommand{\labelenumi}{(\roman{enumi})}
\item $C$ is a nodal curve of arithmetic genus one;
\item $p_1,\ldots,p_n$ are distinct marked smooth points on $C$;
\item $L$ is a line bundle on $C$ of degree $\beta \cdot H$;
\item $v_0, \ldots,v_r$ are global sections of $L \otimes \tildeA$, where $\tildeA := \frA|_C$.
\end{enumerate}
Let $V$ be the open substack of $\frX$ obtained by insisting that the sections $v_0,\ldots,v_r$ have only finitely many basepoints, and let $\varpi \colon V\to \Picsms$ denote the projection.  The stack $V$ is of Deligne--Mumford type.  It carries a perfect obstruction theory relative to $\varpi$ given by the dual to $R^\bullet \pi_* (\frL \otimes \frA)^{\oplus(r+1)}$; thus $V$ is smooth.

To embed $M(\pp)$ into $V$, consider the sheaf $\frE$ on $\frC$ determined by the exact sequence 
\[
\xymatrix{
  0 \ar[r] & \frL\ar[r]^-{ \cdot \sigma} & \frL \otimes \frA \ar[r] & \frE \ar[r] & 0 
}
\]
where the labelled map is multiplication by the section $\sigma$.  The vector bundle
\[
\frN:=\varpi^*\pi_*\frE^{\oplus(r+1)}
\]
over $V$ comes equipped with a tautological section induced by the morphism $\frL \otimes \frA \to \frE$. Let $Z$ be the zero locus of this tautological section. On $Z$, the sections $v_i$ of $L \otimes \tildeA$ are all divisible by our chosen section $\sigma|_C$ of $\tildeA$; let $u_i \in H^0(C,L)$ denote the result of dividing $v_i$ by this chosen section.  Then we have that $M(\pp)$ is the open substack of $Z$ obtained by imposing the stable map non-degeneracy condition: that $u_0, \ldots, u_r$ have no basepoints.

If $m_{\sigma} \colon M(\pp) \to V$ is the embedding just constructed then the complex dual to $[m_{\sigma}^*T_V  \to m_{\sigma}^*\frN]$ gives a perfect obstruction theory for $M(\pp)$ relative to $\varpi$.  This coincides with the usual perfect obstruction theory relative to $\varpi$, which is given by the complex dual to $R^\bullet \pi_* \ev^* O(1)^{\oplus(r+1)}$ where 
\[
\xymatrix{
  C(\pp) \ar[r]^\ev \ar[d]_-{\pi} & \pp^r \\
  M(\pp)
}
\]
is the universal family.
\subsection{The Vakil--Zinger desingularization of $\bar{M}_{1,n}(\pp^r,\beta)$}

We now review the construction, due to Vakil--Zinger~\cite{VZ} and Hu--Li~\cite{huli}, of a partial desingularization $\tildeMP$ of $M(\pp)$.  We give a variant of their construction, which sits in a Cartesian diagram
\begin{equation}\label{blowing up}
  \begin{aligned}
    \xymatrix{\tildeMP\ar[r]^s\ar[d]_{\tilde{\nu}}&M(\pp)\ar[d]^{\nu}\\
      \tildefrPic \ar[r] & \frPic
    }
  \end{aligned}
\end{equation}
where the map $\tildefrPic \to \frPic$ blows up the locus in $\frPic$ where the line bundle has degree zero on the genus-one component.  More precisely, let $\Delta_k$ denote the closure in $\frPic$ of the locus where the line bundle has degree zero on the genus-one component $C_E$, and $C_E$ meets $k$ rational components.  We blow up $\frPic$ along $\Delta_1$, $\Delta_2$,\ldots, in that order, obtaining $\tildefrPic$ as the final blow-up.   Diagram \eqref{blowing up} shows that $\tilde{\nu}$ carries a perfect obstruction theory. By construction $\tildefrPic $ is smooth of dimension $n$ and therefore we can define a virtual class 
\[
\big[\tildeMP\big]^{\vv}=\tilde{\nu}^!\big[\tildefrPic \big].
\]
Moreover $s$ is proper, and so by Costello's push--forward theorem~\cite{costello} 
\[
s_*[\tildeMP]^{\vv}=[M(\pp)]^{\vv}.
\]

\begin{remark}
  The morphism $\tildefrPic \to \frPic$ is the analog for the Picard stack of the Vakil--Zinger weighted blow up $\tildefrM^{\mathrm{wt}}  \to \frM^{\mathrm{wt}}$, where we use notation as in~\cite{huli}.  Indeed $\tildefrPic$ is the base change of $\tildefrM^{\mathrm{wt}}  \to \frM^{\mathrm{wt}}$ along the forgetful morphism $\frPic \to \frM^{\mathrm{wt}}$.
\end{remark}

The space $\tildeMP$ has a main component denoted by $\tildeMP^0$, where the generic map has smooth domain, and other components $\tildeMP^\lambda$, $\lambda \in I$; Hu--Li refer to the $\tildeMP^\lambda$, $\lambda \in I$, as `ghost components'.  Generic stable maps in a ghost component $\tildeMP^\lambda$ have a well-defined number of rational components that meet the elliptic component $C_E$, and so $\tildeMP^\lambda$ sits over $\Delta_k$ for some unique $k$; we refer to this number $k$ of rational components as $k(\lambda)$.  Let $\tildeV = V\times_{\frPic}\tildefrPic$. Then $\tildeV$ is smooth and we have a Cartesian diagram
\begin{equation*}
\xymatrix{\tildeMP \ar[r]\ar[d]&\tildeV\ar[d]\\
M(\pp) \ar[r]^-{m_{\sigma}}&V}
\end{equation*}
such that
\[
[\tildeMP]^{\vv}=m_\sigma^![\tildeV].
\]

\begin{proposition} \label{main_is_smooth} The stack $\tildeMP^0$ is smooth.
\end{proposition}

\begin{proof}
 See ~\cite[Theorem~1.1]{VZ}; it also follows from~\cite[Theorem 2.18]{huli} or the description in \cite{rspw}. 
\end{proof}

Let $X\hookrightarrow \pp^r$ be a smooth projective variety, and let $M(X) = \bar{M}_{1,n}(X,\beta)$. Let 
\begin{align*}
  \tildeMX &=\tildeMP \times_{M(\pp)} M(X), &&
  \tildeMX^0 =\tildeMP^0 \times_{M(\pp)} M(X) \\
  \intertext{and} 
  \tildeMX^\lambda &= \tildeMP^\lambda \times_{M(\pp)} M(X), && \lambda \in I.
\end{align*}
Let $N$ denote the normal bundle $N_{X/\pp^r}$, and let $j \colon \tildeMX \to \tildeMP$ denote the morphism induced by the embedding of $X$ in $\pp^r$.

\begin{proposition} \label{Nlambda} The total space of the sheaf $\pi_*\ev^*N$ on $\tildeMX$ has components $N^{\lambda}$, $\lambda \in I \cup \{0\}$, where $N^\lambda$ is the total space of a vector bundle over $\tildeMX^\lambda$.  The rank of the vector bundle $N^\lambda$ is
\[
\begin{cases}
  \beta \cdot \deg X & \lambda = 0 \\
  \beta \cdot \deg X + \codim X & \text{otherwise.}
\end{cases}
\]
\end{proposition} 

\begin{proof} See~\cite[Theorem~2.10]{huli} and \cite[Theorem~1.2]{VZ}. Let $C_{\pp}\to \bar{C}_{\pp}$ be the contraction of elliptic tails of the universal curve restricted to $\tildeMP^0$ and let $\bar{\pi}_X:\bar{C}_{X}\to \tildeMX^0$ be the projection and $\bar{ev}:\bar{C}_{X}\to X$ the evaluation. By cohomology and base change we have that $(\bar{\pi}_X)_*\bar{\ev}^*N$ is a vector bundle. The morphism $C_{\pp}\to \bar{C}_{\pp}$ 
induces a morphism \[(\bar{\pi}_X)_*\bar{\ev}^*N\to \pi_*\ev^*N|_{\tildeMX^0}.\]
This morphism is injective in all fibres and thus $(\bar{\pi}_X)_*\bar{\ev}^*N$ is a component of the total space of the sheaf $\pi_*\ev^*N$. 
Let $\pi_{\lambda}:C^{\lambda}\to \tildeMX^{\lambda}$ be the universal curve. On $\tildeMX^{\lambda}$ we have that cohomology commutes with base change, so in order to show that $\pi_{\lambda*}f^*N$ is a vector bundle, we show that $H^0(C,f^*N)$ is constant for all $(C,f)\in \tildeMX^{\lambda}$. This follows from the fact that $M_{0,n}(\pp^r,d)$ is convex. Arguing much as in \cite[Theorem~2.18]{huli}, we see that the restriction of $\pi_*\ev^*N$ to $\tildeMX^{\lambda}$ is a component of $\pi_*\ev^*N$.  
\end{proof}

\begin{definition} \label{reduced GW} Set $[\tildeMX^0]^{\vv}=j_{N^0}^![\tildeMP^0]$.
  Let $\gamma_i\in A^{k_i}(X)$ be such that $\sum_{i=1}^nk_i=\vdim\bar{M}_{1,n}(X,\beta) = \vdim \tildeMX^0$.  \emph{Reduced GW invariants} of $X$ are intersection numbers of the form
  \[
  \GW^{X,red}_{1,\beta}(\gamma_1,\cdots,\gamma_n)= [\tildeMX^0]^{\vv} \cdot \prod_{i=1}^n \ev_i^*\gamma_i.
  \]
\end{definition}

\begin{remark} \label{MXP} 
  Later in the paper we will consider the substack $\tildeMXP$ of $\bigcup_{\lambda \in I} \tildeMP^\lambda$ defined by insisting that the collapsed elliptic component maps to $X$.  More precisely, we have that $\tildeMP^{\lambda}\simeq M_{1,k}\times_{\pp}\tilde{M}_{0,n_1+1}(\pp) \times_{\pp}\cdots \times_{\pp}\tilde{M}_{0,n_k+1}(\pp)$. Fix one of the nodes, say the one corresponding to the last marked point of $\tilde{M}_{0,n_1+1}(\pp)$. Let $q:\tildeMP^{\lambda}\to \tilde{M}_{0,n_1+1}(\pp)$ be the projection, and set
  \begin{align*}
    \tildeMXP^\lambda = q^{-1} \left( ev_{n_1+1}^{-1}X\right), && \tildeMXP = \bigcup_{\lambda \in I} \tildeMXP^\lambda.
  \end{align*}
  Note that $\tildeMX^\lambda$ is contained in $\tildeMXP^\lambda$.  The vector bundles $N^\lambda \to \tildeMX^\lambda$ extend to vector bundles on $\tildeMXP^\lambda$, which we also denote by $N^\lambda$, and  Proposition~\ref{Nlambda} holds over $\tildeMXP = \bigcup_\lambda \tildeMXP^\lambda$ too.  The proof is the same.
\end{remark}

\begin{remark}
  Let $F$ denote the kernel of the surjective morphism of vector bundles $O_X(1)^{\oplus(r+1)} \to N$ on $X$.  Then $R^\bullet \pi_* f^* F$ is a perfect dual obstruction theory for $\tildeMX$ relative to $\tildefrPic$, and therefore defines a virtual class $[ \tildeMX ]^\vv$.  This virtual class is in fact intrinsic to $X$, i.e.~independent of the choice of embedding $X \hookrightarrow \pp^r$.  In Definition~\ref{reduced GW} above we defined a virtual class $[ \tildeMX^0 ]^\vv$ on the main component of $\tildeMX$.  In rest of this paper we will define virtual classes $[ \tildeMX^\lambda ]^\vv$ on the ghost components $\tildeMX^\lambda$, $\lambda \in I$, in such a way that
  \[
    \left[ \tildeMX \right]^\vv = 
    \left[ \tildeMX^0 \right]^\vv +
    \sum_{\lambda \in I} \left[ \tildeMX^\lambda \right]^\vv.
  \]
\end{remark}

\section{A further blow-up}

It will be convenient to work with blow-ups $\BltildeMP$ of $\tildeMP$ and $\BltildeV$ of $\tildeV$ with the property that the cone $C_{\BltildeMP/\BltildeV}$ is a line bundle.  In this section we introduce these blow-ups.

\begin{construction} Let $\tildeA$ be a very ample line bundle as in Section \ref{cfk}, let $V$ denote the total space of $\pi_*L(\tildeA)$, and let $\tildeV=V\times_{\frPic}\tildefrPic$. We have an embedding $\tildeMP\to \tildeV$.  Recall that the main component of $\tildeMP$ is denoted by $\tildeMP^0$, and that the remaining components  $M(\pp)^\lambda$ are indexed by $\lambda \in I$.  We fix an order for $\lambda\in I$. We blow up $\tildeV$ along $\tildeMP^0$ and then along $\tildeMP^{\lambda}$, one component at a time in the given order, denoting the resulting blow up by $p_{\tildeV} \colon \BltildeV \to \tildeV$. Note that $\BltildeV$ is smooth. Let $\BltildeMP$ be the exceptional divisor in $\BltildeV$. Then there is a Cartesian diagram
\[\xymatrix{\BltildeMP\ar[r]\ar[d]_{p_{\pp}}&\BltildeV\ar[d]^{p_{\tildeV}}\\
\tildeMP\ar[r]&\tildeV}
\]
and $\BltildeMP$ has components $\BltildeMP^\lambda$, $\lambda \in \{0\} \cup I$. 
\end{construction}
\begin{notation} 
Let $\BltildeMX$ denote the fiber product $\tildeMX\times_{\tildeMP}\BltildeMP$, and let $p_{X}:\BltildeMX\to \tildeMX$ be the morphism induced by $p_{\pp}$.   Recall the definition of $\tildeMXP$ from Remark~\ref{MXP}.  We denote by $\BltildeMXP$ the fibre product $\tildeMXP\times_{\tildeMP}\BltildeMP$.
\end{notation}

\begin{proposition} \label{definevirtualclasses} \ 

(i) $$E^\bullet_{\BltildeMP/\tildefrPic}:=[T_{\BltildeV/\tildefrPic}\to p_{\pp}^* \pi_* L(\tildeA)|^{\oplus(r+1)}_{\tildeA}]$$ is a perfect dual obstruction theory for $\BltildeMP$ relative to $\tildefrPic$. 

(ii) 
\[
E^\bullet_{\BltildeMX/\tildefrPic}:=[T_{\BltildeV/\tildefrPic}\to p_X^*\pi_* L(\tildeA)|^{\oplus(r+1)}_{\tildeA}\oplus p_X^*\pi_* f^*N(\tildeA)\to p_X^*\pi_* f^*N(\tildeA)|_{\tildeA}]
\]
is a perfect dual obstruction theory for $\BltildeMX$ relative to $\tildefrPic$. 
\end{proposition}
\begin{proof} The vector bundle $\pi_* L(\tildeA)|^{\oplus(r+1)}_{\tildeA}$ is a dual obstruction theory for the embedding  $\tildeMP\to \tildeV$ and therefore the vector bundle $ p_{\pp}^* \pi_* L(\tildeA)|^{\oplus(r+1)}_{\tildeA} $ is a perfect dual obstruction theory for $\BltildeMP\to \BltildeV$.  The composition $$T_{\BltildeV/\tildefrPic}\to p_{\pp}^*T_{\tildeV/\tildefrPic}\to p_{\pp}^* \pi_* L(\tildeA)|^{\oplus(r+1)}_{\tildeA}$$ gives a complex $[T_{\BltildeV/\tildefrPic}\to p_{\pp}^* \pi_* L(\tildeA)|^{\oplus(r+1)}_{\tildeA}]$.   This proves (i). 

The complex in (ii) is supported in $[0,2]$, but the surjectivity of the map $\pi_* L(\tildeA)|^{\oplus(r+1)}_{\tildeA}\to \pi_* f^*N(\tildeA)|_{\tildeA}$ implies that it is in fact supported in $[0,1]$.  Now argue as in (i).
\end{proof}

Proposition~\ref{definevirtualclasses} defines virtual classes on $\BltildeMP$ and $\BltildeMX$.

\begin{lemma} \label{pushforwards} We have that $$(p_{\pp})_*[\BltildeMP]^{\vv}=[\tildeMP]^{\vv}$$ and  $$(p_X)_*[\BltildeMX]^{\vv}=[\tildeMX]^{\vv}.$$
\end{lemma}
\begin{proof} Pull-backs commute with push-forwards.
\end{proof}
\begin{lemma} \label{bl to tilde}We have
\[
(p_{\pp})_*[\BltildeMP]^{\vv}=[\tildeMP^0]+(p_{\pp})_*\sum_{\lambda \in I}0_{\frE_{\BltildeMP/\tildefrPic}}^![C^{\lambda}_{\BltildeMP/\tildefrPic}]
\]
where $\frE_{\BltildeMP/\tildefrPic}$ is the vector bundle stack $h^1/h^0(E^\bullet_{\BltildeMP/\tildefrPic})$.
\end{lemma}
\begin{proof} Since
\[
[\BltildeMP]^{\vv}=0_{\frE_{\BltildeMP/\tildefrPic}}^![C_{\BltildeMP/\tildefrPic}].
\]
we need to show that 
\begin{equation}
  \label{MP0 excess}
  (p_{\pp})_* 0_{\frE_{\BltildeMP/\tildefrPic}}^![C^{0}_{\BltildeMP/\tildefrPic}]=[\tildeMP^{0}].
\end{equation}
Since $C^{0}_{\BltildeMP/\tildefrPic}$ is the only component of $C_{\BltildeMP/\tildefrPic}$ supported on $\BltildeMP^0$ we have that
\[
(p_{\pp})_* 0_{\frE_{\BltildeMP/\tildefrPic}}^![C^{0}_{\BltildeMP/\tildefrPic}]=K[\tildeMP^{0}]
\]
for some $K\in\Q$. Lemma~\ref{pushforwards} gives 
\[
(p_{\pp})_*[\BltildeMP]^{\vv}=[\tildeMP]^{\vv}
\]
and therefore $K=1$.
\end{proof}

\section{A virtual class of the correct dimension}

Let $\lambda \in I$.  The virtual class $\left[\BltildeMP\right]^{\vv}$, when restricted to $\BltildeMP^\lambda$ and pulled back using $N^\lambda$, will have dimension different from the virtual dimension of $\BltildeMX$.  In this section we define a new virtual class $\left[\BltildeMP\right]_X^{\vv}$ on $\BltildeMP$, which restricts and pulls back to something of the correct dimension.  We will use this class (in \S\ref{sec:main theorem} below) to define virtual classes on each of the components $\BltildeMX^\lambda$.

\subsection{Reduced restricted obstruction theories on $\BltildeMXP^{\lambda}$}

\begin{lemma}
(i) On $\tildeMXP^{\lambda}$, we have a surjective morphism $$E_{\tildeMP/\tildefrPic}|_{\tildeMXP^{\lambda}}\to h^1(\pi_*f^*N)|_{\tildeMXP^{\lambda}}.$$

(ii) On $\BltildeMXP^{\lambda}$, we have a surjective morphism $$E_{\BltildeMP/\tildefrPic}|_{\BltildeMXP^{\lambda}}\to p_{X}^* h^1(\pi_*f^*N)|_{\BltildeMXP^{\lambda}}.$$

\end{lemma}
\begin{proof} We have that $E_{\tildeMP/\tildefrPic}=[T_{\tildeV/\tildefrPic}\to \pi_* L(\tildeA)|_{\tildeA}^{\oplus(r+1)}]$. From the direct sum of $r+1$ copies of the exact sequence
\[
0\to \pi_* L\to \pi_* L(\tildeA)\to \pi_* L(\tildeA)|_{\tildeA} \to R^1\pi_* L\to 0
\]
we get surjective morphisms 
\begin{equation}
  \label{eq:twoheadedarrows}
  \pi_* L(\tildeA)|_{\tildeA}^{\oplus(r+1)} \twoheadrightarrow R^1\pi_* L^{\oplus(r+1)}\twoheadrightarrow R^1\pi_*f^*N
\end{equation}
on $\tildeMXP$. Thus we get a morphism of two term complexes $$[T_{\tildeV/\tildefrPic}|_{\tildeMXP^{\lambda}}\to E_{\tildeMP/\tildeV}|_{\tildeMXP^{\lambda}}]\to [0\to  R^1\pi_*f^*N  ]$$ supported in $[0,1]$. This shows the existence of the morphism in (i); surjectivity follows from the surjectivity of \eqref{eq:twoheadedarrows}.

We have that $E_{\BltildeMP/\tildefrPic}=[T_{\BltildeV/\tildefrPic}\to p_{\tildeV}^* \pi_* L(\tildeA)|_{\tildeA}^{\oplus(r+1)}]$.  Thus $E_{\BltildeMP/\tildefrPic}$ is obtained by pulling back $[T_{\tildeV/\tildefrPic} \to E_{\tildeMP/\tildeV}]$ along $p_{\tildeV}$ and composing with $T_{\BltildeV/\tildefrPic} \to p_{\tildeV}^* T_{\tildeV/\tildefrPic}$.  Thus (i) implies the existence of the morphism in (ii).  It remains to prove surjectivity.  Consider the exact sequence
\[0\to T_{\BltildeV/\tildefrPic}\to p_{\tildeV/\tildefrPic}^* T_{\tildeV}\to i_*Q\to 0,
\]
where $Q$ is a sheaf on the exceptional divisor. Restricting to $\BltildeMP$ we get
\[0\to G\to T_{\BltildeV/\tildefrPic}|_{\BltildeMP}\to p_{\tildeV/\tildefrPic}^* T_{\tildeV}|_{\BltildeMP}\to i_*Q|_{\BltildeMP}\to 0.
\]
This gives two exact sequences
\[0\to G\to T_{\BltildeV/\tildefrPic}|_{\BltildeMP}\to H\to 0
\] and 
\[0\to H\to p_{\tildeV}^* T_{\tildeV/\tildefrPic}|_{\BltildeMP}\to i_*Q|_{\BltildeMP}\to 0.
\]
We now consider the commutative diagram 
\[\xymatrix{0\ar[r]&G\ar[r]\ar[d]& T_{\BltildeV/\tildefrPic}|_{\BltildeMP}\ar[r]\ar[d]&H\ar[r]\ar[d]&0\\
0\ar[r]&0\ar[r]&p_{\tildeV}^* E_{\tildeMP/\tildeV}|_{\BltildeMP}\ar[r]\ar[d]&p_{\tildeV}^* E_{\tildeMP/\tildeV}|_{\BltildeMP}\ar[r]\ar[d]&0\\
&&\widetilde{Obs}\ar[r]&F
}
\]
Here $\widetilde{Obs}$ is the cokernel of $T_{\BltildeV/\tildefrPic}|_{\BltildeMP}\to p_{\tildeV}^* E_{\tildeMP/\tildeV}|_{\BltildeMP}$ and $F$ is the cokernel of $H\to p_{\tildeV}^* E_{\tildeMP/\tildeV}|_{\BltildeMP}$. The Snake Lemma gives $F\simeq \widetilde{Obs}$.
Similarly, consider 
the commutative diagram 
\[\xymatrix{0\ar[r]&H\ar[r]\ar[d]& p_{\tildeV}^*T_{\tildeV/\tildefrPic}|_{\BltildeMP}\ar[r]\ar[d]& i_*Q|_{\BltildeMP}\ar[r]\ar[d]&0\\
0\ar[r]&p_{\tildeV}^* E_{\tildeMP/\tildeV}|_{\BltildeMP}\ar[r]\ar[d]&p_{\tildeV}^* E_{\tildeMP/\tildeV}|_{\BltildeMP}\ar[r]\ar[d]&0\ar[r]&0\\
&F\ar[r]&p_{\tildeV}^* R^1\pi_* L^{\oplus(r+1)}|_{\BltildeMP}
}
\]
The Snake Lemma gives that $F\to p_{\tildeV}^* R^1\pi_* L^{\oplus(r+1)}|_{\BltildeMP}$ is surjective. Since $F\simeq \widetilde{Obs}$ we get that the morphisms 
\[  p_{\tildeV}^* E_{\tildeMP/\tildeV}|_{\BltildeMP} \twoheadrightarrow \widetilde{Obs}\twoheadrightarrow p_{\tildeV}^* R^1\pi_* L^{\oplus(r+1)}|_{\BltildeMP}\]
are surjective.  Restricting to $\BltildeMXP$ and composing with the right-hand morphism in \eqref{eq:twoheadedarrows} proves (ii).
\end{proof}

\begin{lemma} \label{reduced obstruction} 

(i) Let $\lambda \in I$, let $E^{X,\lambda}_{\BltildeMP/\tildefrPic}$ denote the kernel of 
\[
  E_{\BltildeMP/\tildefrPic}|_{\BltildeMXP^{\lambda}}\to p_{X}^* h^1(\pi_*f^*N)
\]
and let $E^{X,\lambda}_{\BltildeMXP/\tildefrPic}$ denote the kernel of 
\[
  E_{\BltildeMXP/\tildefrPic}|_{\BltildeMXP^{\lambda}}\to p_{X}^* h^1(\pi_*f^*N)
\]
Then $E^{X,\lambda}_{\BltildeMP/\tildefrPic}$ is a vector bundle stack on $\BltildeMXP^{\lambda}$ that contains $C^{\lambda}_{\BltildeMP/\tildefrPic}|_{\BltildeMXP^{\lambda}}$.

(ii) Let $Q$ be a fixed node as in Remark~\ref{MXP}. Denote by $E^{X,\lambda}_{\BltildeMXP/\tildefrPic}$ the vector bundle stack $E^{X,\lambda}_{\BltildeMP/\tildefrPic}\oplus \pi_*f^*N|_Q$ and by $C^{X,\lambda}_{\BltildeMXP/\tildefrPic}$ the cone stack stack $C^{X,\lambda}_{\BltildeMP/\tildefrPic}\oplus \pi_*f^*N|_Q$. Then we have an embedding $$C^{X,\lambda}_{\BltildeMXP/\tildefrPic}\hookrightarrow E^{X,\lambda}_{\BltildeMP/\tildefrPic}\oplus \pi_*f^*N|_Q.$$
\end{lemma}

\begin{proof}
Let $U$ be an open subset of the smooth locus in $\tildeMP^{\lambda}$, and write 
\begin{align*}
  Bl(U)=p_{\pp}^{-1}U, && Bl^X(U) = \BltildeMXP\cap Bl(U).
\end{align*}
We first show the statement for $C^{\lambda}_{\BltildeMP/\tildefrPic}|_{Bl(U)}$, which is a vector bundle stack over $Bl(U)$. This amounts to showing that the morphism 
\[
C^{\lambda}_{\BltildeMP/\tildefrPic}|_{Bl^X(U)}\to p_{X}^*h^1(\pi_*f^*N)|_{Bl(U)}
\]
is zero.  Recall the construction of $\tildefrPic$ discussed just below equation~\ref{blowing up} and the definition of $k(\lambda)$ just above Proposition~\ref{main_is_smooth}.  Set $k = k(\lambda)$, so that the component $\tildeMP^{\lambda}$ sits over the locus $\Delta_k$ in $\frPic$.

Case 1: $k = 1$. We want to show that the composition
\[
  C^{\lambda}_{\BltildeMP/\tildefrPic}|_{Bl^X(U)}\to p_{\pp}^*C^{\lambda}_{\tildeMP/\tildefrPic}|_{Bl^X(U)}\to 
  p_{\pp}^* \left( \frac{E_{\tildeMP/\tildeV}}{T_{\tildeV/\tildefrPic}}\right)|_{Bl^X(U)}\to p_{X}^*h^1(N)|_{Bl^X(U)} 
\]
is zero. For this it is enough to show that 
\begin{equation}\label{zerocomp} C^{\lambda}_{\tildeMP/\tildefrPic}|_{\tildeMXP\cap U}\to \frac{E_{\tildeMP/\tildeV}}{T_{\tildeV/\tildefrPic}}|_{\tildeMXP\cap U}\to h^1(N)|_{\tildeMXP\cap U}
\end{equation}
is zero.
We have that 
\[
C^{\lambda}_{\tildeMP/\tildefrPic}\simeq \frac{C^{\lambda}_{\tildeMP/\pi_* L(\tildeA)^{\oplus (r+1)}}}{\pi_*f^* L(\tildeA)^{\oplus (r+1)}}.
\]
After shrinking $U$, if necessary, and passing to an \'{e}tale cover we may assume that $\tildeA=A_1+A$, where $A$ is a section of the universal curve over $\tildefrPic$, and that there exists some $i$ in $\{0,1,\ldots,r\}$ such that $u_i$ does not vanish on the contracted elliptic component of the curve.  Without loss of generality we take $i=0$.  Let $\tildeW$ be the total space of $\pi_*L(A)^{\oplus (r+1)}$ over $\tildefrPic$. Then the construction in \S\ref{cfk} gives a composition of embeddings
\[ 
U\to \tildeW\to \tildeV
\] 
With this notation, $C^{\lambda}_{U/\tildefrPic}$ is locally isomorphic to $\frac{C_{U/\tildeW}\oplus C_{\tildeW/\tildeV}}{T_{\tildeV/\tildefrPic}}$.  Let us describe $C_{U/\tildeW}$ on an open set; we will see in particular that $U$ is codimension two in $\tildeW$. Consider now the local charts introduced in \cite{huli}. Let $\frM^{div,d}$ be the Artin stack of pairs $(C,D)$ where $C$ is a nodal elliptic curve and $D \subset C$ is an effective divisor of degree~$d$; this was denoted in~\cite{huli} by $\frD_1^d$.  Let $\tildeW'$ be the total space $\oplus_{i=1}^r \pi_*L(A)$ over $\frM^{div,d}$ -- noting that the index $i$ runs from $1$ not $0$ -- and observe that $\tildeW$ is the total space of $\oplus_{i=1}^r \pi_*L(A)$ over $\frM^{div,d+1}$. We have embeddings 
\[\xymatrix {U\ar[r]& \tildeW'\ar[r]\ar[d]& \tildeW\ar[d]\\
&\frM^{div,d}\ar[r]&\frM^{div,d+1}
}
\]
where the square is Cartesian and the bottom horizontal map is given by $A$.
The normal bundle of $\frM^{div,d}$ in $\frM^{div,d+1}$ is isomorphic to $\pi_* O(A)|_A$.  The Euler sequence on $\pp^r$ implies that the following sequence is exact
\[0\to \pi_* O(A)|_A\to \pi_*L(A)|_A^{\oplus (r+1)}\to  \pi_* f^* T_{\pp}(A)|_A\to 0
\]
By further shrinking $U$ we may assume that this sequence splits on $U$; this shows that $\pi_* f^* T_{\pp}(A)|_A$ is a dual obstruction theory for $U\to \tildeW'$.

Let $O(\xi)$ be the line bundle over $\tildeMP^\lambda$ with fiber at a point $(C,f)$ the space $T^E_Q\otimes T^R_Q$, where $Q$ is the node of $C$ connecting the contracted genus one component $E$ of $C$ to the rational part $R$.  With this we have that $C_{U/\tildeW'}$ is isomorphic to $O(\xi)$. This implies that
\[
C_{U/\tildefrPic}\simeq\frac{O(\xi)\oplus\pi_*O(A)|_A\oplus \pi_*L(A_1)|_{A_1}^{\oplus (r+1)}}{ \pi_*L(A+A_1)^{\oplus (r+1)}}.
\]
With this, (\ref{zerocomp}) becomes
\begin{align*}
  \frac{O(\xi)\oplus\pi_*O(A)|_A\oplus\pi_* L(A_1)|_{A_1}^{\oplus (r+1)}}{\pi_*L(A+A_1)^{\oplus (r+1)}} &\to 
  \frac{\pi_*T_{\pp}(A)|_{A}\oplus\pi_*O(A)|_A\oplus \pi_* L(A_1)|_{A_1}^{\oplus (r+1)}}{\pi_*L(A+A_1)^{\oplus (r+1)}} \\
  &\to\pi_*f^*N(A)|_{A}.
\end{align*}
(This sequence lives on $\tildeMXP \cap U$.)  Using $\pi_*O(A)|_A\simeq \E^{\vee}$, where $\E$ is the Hodge bundle, we can rewrite (\ref{zerocomp}) as
\begin{align*}
  \frac{ T^R_Q\otimes \E^{\vee}\oplus \E^{\vee}\oplus\pi_* L(A_1)|_{A_1}^{\oplus (r+1)}}{\pi_*L(A+A_1)^{\oplus (r+1)}} &\to
  \frac{\pi_*T_{\pp}|_{Q}\otimes\E^{\vee}\oplus\E^\vee\oplus \pi_* L(A_1)|_{A_1}^{\oplus (r+1)}}{\pi_*L(A+A_1)^{\oplus (r+1)}} \\ 
  &\to\pi_*f^*N|_Q\otimes \E^{\vee}.
\end{align*}
(This sequence also lives on $\tildeMXP \cap U$.)  Since $Q\in X$ we have an exact sequence
\begin{equation*} 0\to f_R^* T_X|_Q\otimes \E^{\vee}\to f_R^*T_{\pp}|_Q\otimes \E^{\vee}\to f_R^* N|_Q\otimes \E^{\vee}\to 0
\end{equation*}
where $f_R$ is the restriction of $f$ to $R$.  This shows that the composition in (\ref{zerocomp}) is zero, as claimed.

Case 2: $k > 1$.    Let $\Xi_k$ denote the exceptional divisor of $\tildefrPic\to \frPic$ which maps to $\Delta_k$.  We have that
\[
C^{\lambda}_{\BltildeMP/\tildefrPic}|_{Bl(U)} = p_{\pp}^* \,O(\Xi_k).
\]
Arguing as in Case~1, but replacing $O(\xi)$ by $O(\Xi_k)$, we see that we need to show that the composition 
\begin{equation} \label{composition_case_2}
O(\Xi_k)\to \pi_* f^* T_{\pp}(A)|_A\to \pi_* f^* N(A)|_A
\end{equation}
is zero.  (This sequence lives on $\tildeMXP \cap U$.)  Choose co-ordinates
\[
(C; p_1,\ldots,p_n; L; u_0,\ldots,u_r; t_1,\ldots,t_k)
\]
on $U$, where $(C; p_1,\ldots,p_n; L; u_0,\ldots,u_r)$ are as in \S\ref{cfk} and $t_1,\ldots,t_k$ are projective co-ordinates on the normal bundle to $\Delta_i$ in $\frPic$.  The tautological sequence on the projectivised normal bundle starts with
\[
0 \to O(\Xi_k) \to \bigoplus_{i=1}^k \E^{\vee} \otimes T_{Q_i} R_i 
\]
where the rational component $R_i$ meets the contracted genus one component $E$ at $Q_i$.  Hu--Li have shown (in the statement and proof of \cite[Theorem~2.19]{huli}) that the morphism $O(\Xi_k)\to \pi_* f^* T_{\pp}(A)|_A$ in \eqref{composition_case_2} arises from the composition of the exact sequence above with the map
\[
\bigoplus_{i=1}^k T_{Q_i} R_i \to f^* T_{\pp} \big|_{Q_i}
\]
induced by the derivative of the section that defines $X$.  Note that the right-hand side here does not in fact depend on $i$: the restrictions of $f^* T_{\pp}$ to $Q_i$ coincide for $i = 1,2,\ldots,k$, since the elliptic component $E$ of the curve is collapsed by $f$.  It follows that the composition \eqref{composition_case_2} is zero.

Remainder of the argument: So far we have shown that the composition
\[ 
  C^{\lambda}_{\BltildeMP/\tildefrPic}|_{\BltildeMXP^{\lambda}}\to E_{\BltildeMP/\tildefrPic}|_{\BltildeMXP^{\lambda}}\to p_X^* h^1(\pi_*f^*N)
\] 
is zero on an open set $Bl(U)$.  Since $C^{\lambda}_{\BltildeMP/\tildefrPic}$ is a line bundle and $\BltildeMXP^{\lambda}$ is irreducible, this in fact shows that the composition is zero on all of $\BltildeMXP^{\lambda}$.  This proves the Lemma.

Statement (ii) is an immediate consequence of (i).
\end{proof}

\subsection{A new virtual class on $\BltildeMP$}

We are now in a position to define the new virtual class $\left[\BltildeMP\right]_X^{\vv}$ on $\BltildeMP$.

\begin{definition}
  Let 
  \begin{align*}
  \frE^0 & \text{ denote the vector bundle stack associated to $E^\bullet_{\BltildeMP/\tildefrPic}|_{\BltildeMP^0}$} \\
  \frE^{X,\lambda} & \text{ denote the vector bundle stack associated to the complex $E_{\BltildeMXP/\tildefrPic}^{X,\lambda}$} \\
    C^0 & \text{ denote the component of $C_{\BltildeMP/\tildefrPic}$ supported on $\BltildeMP^0$} \\
    C^{X,\lambda} & \text{ denote the component of $C^\lambda_{\BltildeMXP/\tildefrPic}$ supported on $\BltildeMXP^\lambda$} 
  \end{align*}
  Recall that $C^0$ is contained in $\frE^0$, and that $C^\lambda$ is contained in $\frE^{X,\lambda}$.  We define a virtual class on $\BltildeMP$ by:
  \begin{align*}
    & \left[\BltildeMP^0\right]^{\vv} = 0^!_{\frE^0} \left[C^0\right] \\
    & \left[\BltildeMP^\lambda\right]_X^{\vv} = 0^!_{\frE^{X,\lambda}} \left[C^{X,\lambda}\right] \\
    & \left[\BltildeMP\right]_X^{\vv} =  \left[\BltildeMP^0\right]^{\vv} + \sum_{\lambda \in I} \left[\BltildeMP^\lambda\right]_X^{\vv} \in A_* \left(\BltildeMP\right).
  \end{align*} 
\end{definition}

\begin{remark}
  In general, $[\BltildeMP]_X^{\vv}$ is not a pure-dimensional cycle.
\end{remark}

\begin{remark} \label{define Cok}
  Since the cones $C^0$ and $C^\lambda$ are vector bundle stacks, we can write the virtual class in terms of excess bundles:
  \begin{align*}
    & \left[\BltildeMP^0\right]^{\vv} = \ctop (\Cok^0) \cdot \left[\BltildeMP^0\right] \\
    & \left[\BltildeMP^\lambda\right]_X^{\vv} = \ctop(\Cok^{\lambda}) \cdot \left[\BltildeMXP^{\lambda}\right]
  \end{align*} 
  where $\Cok^0$ is the cokernel of $C^0 \to \frE^0$ and $\Cok^\lambda$, $\lambda \in I$, is the cokernel of $C^{X,\lambda} \to \frE^{X,\lambda}$.  Each $\Cok^\lambda$ is a vector bundle, $\lambda \in \{0\} \cup I$.
\end{remark}

\section{Deformations of obstruction theories and cone stacks}\label{deformations}

We now give a a construction inspired by Kim--Kresch--Pantev's proof of functoriality for virtual classes~\cite{kkp}. 

\begin{construction} \label{con:deformation}
  Consider an exact triangle of complexes on an algebraic stack which are supported in degrees [-1,0]
  \begin{equation}
    \label{eq:distinguished_triangle}
    E^{\bullet}\stackrel{\phi}{\rightarrow} F^{\bullet}\to G^{\bullet}\to E^{\bullet}[1].
  \end{equation}
  We construct $\mathcal{F}$ a flat family over $\pp^1$ defined as follows. Take the morphism $$g:E^{\bullet}\otimes O_{\pp^1}(-1)\to E^{\bullet}\oplus F^{\bullet}$$ with $g=(T\cdot id, U\cdot \phi)$. Here $T$ and $U$ are coordinates on $\pp^1$. Define $\mathcal{F}$ to be the vector bundle stack associated to the cokernel of $g$. The general fiber of $\mathcal{F}$ is isomorphic to $h^1/h^0(F^{\bullet})^{\vee}$ and the special fiber $\mathcal{F}^0$ is isomorphic to $h^1/h^0(E^{\bullet})\oplus h^1/h^0(G^{\bullet})$. We call $h^1/h^0(E^{\bullet})\oplus h^1/h^0(G^{\bullet})$ the \emph{limit} of $h^1/h^0(F^{\bullet})$ and we use the following notation
  \[
    h^1/h^0(F^{\bullet})^{\vee} \rightsquigarrow h^1/h^0(E^{\bullet})^{\vee}\oplus h^1/h^0(G^{\bullet})^{\vee}.
  \]
\end{construction}

\begin{proposition}
The family $\mathcal{F}$ is trivial over $\A^1\simeq\pp^1\backslash [0:1]$.
\end{proposition}
\begin{proof} Consider the automorphism $$\psi: h^1/h^0(E^{\bullet})^{\vee}\oplus h^1/h^0(F^{\bullet})^{\vee}\to h^1/h^0(E^{\bullet})^{\vee}\oplus h^1/h^0(F^{\bullet})^{\vee}$$ defined by $\psi(x,y)=(x, y-U\cdot \phi(x))$. The cokernel of $g$ is isomorphic to the cokernel of $\psi\circ g: h^1/h^0(E^{\bullet})^{\vee}\to h^1/h^0(E^{\bullet})^{\vee}\oplus h^1/h^0(F^{\bullet})^{\vee}$. As $\psi \circ g$ is $(T\cdot id,0)$ we see that its cokernel is the trivial family $\A^1\times F$. The associated vector bundle stack is $\A^1\times h^1/h^0(F^{\bullet})^{\vee}$.
\end{proof}
\begin{construction}  \label{con:limiting_bundle}  In notation as above we have
\[h^1/h^0(F^{\bullet})^{\vee} \rightsquigarrow h^1/h^0(E^{\bullet})^{\vee}\oplus h^1/h^0(G^{\bullet})^{\vee}.
\]
Given $C$ a cone stack inside a vector bundle stack $h^1/h^0(F^{\bullet})^{\vee}$ and $E^{\bullet}$, $F^{\bullet}$ complexes as above. By the above proposition we can consider $C\times \A^1$ in $\A^1\times h^1/h^0(F^{\bullet})^{\vee}$. Let $\mathcal{C}$ be the closure of $C\times \A^1$ in $\mathcal{F}$. Then $\mathcal{C}$ is a flat family. We denote the fiber over $[0:1]$ by $C^0$ and we call it the limit of $C$ in $h^1/h^0(E^{\bullet})^{\vee}\oplus h^1/h^0(G^{\bullet})^{\vee}$, writing
\[C \rightsquigarrow C^0 .
\]
\end{construction}

\medskip

Consider now the special case of Construction~\ref{con:limiting_bundle} where $X\to Y\to Z$ are DM morphisms of stacks, $E^{\bullet}$ is a perfect obstruction theory for $X\to Y$, $F^{\bullet}$ is a perfect obstruction theory for $X\to Z$, $G^{\bullet}$ is a perfect obstruction theory for $Y\to Z$, and $C = C_{X/Z}$.
We thus obtain embeddings of cones
\[
C_{X/Z}\hookrightarrow h^1/h^0(F^{\bullet})^{\vee}
\]
and a special fiber
\[
C^0\hookrightarrow  h^1/h^0(E^{\bullet})^{\vee}\oplus h^1/h^0(G^{\bullet})^{\vee}.
\] 
In the fiber of $\mathcal{F}$ over zero we also have an embedding
\[
C_{X/C_{Y/Z}}\hookrightarrow  h^1/h^0(E^{\bullet})^{\vee}\oplus h^1/h^0(G^{\bullet})^{\vee}.
\] 
 Kim--Kresch--Pantev proved that $[C_0]=[C_{X/C_{Y/Z}}]$ in $A_*(\mathcal{F})$.  However, in general it is not true that the limit of $C_0$ is $C_{X/C_{Y/Z}}$. 

\subsection{A deformation of the obstruction theory on $\BltildeMX$}

Let us now apply the general theory just discussed to the obstruction theory $E_{\BltildeMX/\tildefrPic}$.  This will construct a deformation of the cone $C_{\BltildeMX/\tildefrPic}$.

\begin{notation} Recall that on $X$ we have a surjective morphism from $O(1)^{\oplus(r+1)}$ to $N$.  Let $F$ denote the kernel of this morphism.
\end{notation}
\begin{lemma} \label{lem:N}
  Let $N^0$ and $N^\lambda$, $\lambda \in I$, be as in Remark~\ref{MXP}.  Let $(N^0)^{\bullet}$ be the complex $[0\to N^0]$ supported in $[0,1]$ and $(N^{\lambda})^{\bullet}$ be the complex $[0\to N^{\lambda}]$ supported in $[0,1]$. Then, we have morphisms of complexes on $\tildeMX^0$ and $\tildeMXP^\lambda$:
\begin{align}
  \label{nzerolambda downstairs}(N^0)^{\bullet}\to R^\bullet \pi_* f^* F && (N^{\lambda})^{\bullet}\to R^\bullet \pi_* f^* F
\end{align}
and morphisms of complexes on $\BltildeMX^0$ and $\BltildeMXP^\lambda$:
\begin{align}
  \label{nzerolambda upstairs}p_{\pp}^* (N^0)^{\bullet}\to \BlF^\bullet && p_{\pp}^* (N^{\lambda})^{\bullet}\to  \BlF^\bullet
\end{align}
where $\BlF^\bullet$ is the complex
\[
[T_{\BltildeV/\tildefrPic} \to p_{\pp}^* \left(\pi_*L(\tilde{A})|_{\tilde{A}}^{\oplus (r+1)}\oplus \pi_*f^*N(\tilde{A}) \right)\to p_{\pp}^*\pi_*f^*N(\tilde{A})|_{\tilde{A}}]
\]
supported in $[0,2]$.
\end{lemma}

\begin{remark}
  As we will see in the proof, $\BlF^\bullet$ is in fact quasi-isomorphic to a complex supported in $[0,1]$.
\end{remark}

\begin{proof} We replace $R\pi_* f^* F$ with the quasi-isomorphic complex $$[\pi_*L(\tilde{A})^{\oplus (r+1)}\to \pi_*L(\tilde{A})|_{\tilde{A}}^{\oplus (r+1)}\oplus \pi_*f^*N(\tilde{A})\to \pi_*f^*N(\tilde{A})|_{\tilde{A}}]$$ supported in $[0,2]$. The morphism $N^0\to \pi_*f^*N(\tilde{A})$ given by multiplication by $\sigma$ induces a morphism 
\[\xymatrix{0\ar[r]\ar[d]&\pi_*L(\tilde{A})^{\oplus (r+1)}\ar[d]\\
N^0\ar[r] & \pi_*L(\tilde{A})|_{\tilde{A}}^{\oplus (r+1)}\oplus \pi_*f^*N(\tilde{A})\ar[d]\\
& \pi_*f^*N(\tilde{A})|_{\tilde{A}}
}
\]
The arguments for $N^{\lambda}$, $p_{\pp}^* (N^0)^{\bullet}$, and $p_{\pp}^* (N^\lambda)^{\bullet}$ are the same.
\end{proof}
\begin{lemma} \label{lem:G}
Let $(G^0)^{\bullet}$ and $(G^{\lambda})^{\bullet}$ be the mapping cones of the morphisms in \eqref{nzerolambda downstairs}, and let $\BlGzero^{\bullet}$ and $\BlGlambda^{\bullet}$ be the mapping cones of the morphisms in \eqref{nzerolambda upstairs}. Then $(G^0)^{\bullet}$, $(G^{\lambda})^{\bullet}$, $\BlGzero^{\bullet}$, and $\BlGlambda^{\bullet}$ are perfect complexes concentrated in $[0,1]$.
\end{lemma}
\begin{proof} Let $\lambda \in \{0\} \cup I$.  The mapping cone of the morphism in \eqref{nzerolambda downstairs} is the complex 
\[
[\pi_*L(\tilde{A})^{\oplus (r+1)}\oplus N^\lambda\to \pi_*L(\tilde{A})|_{\tilde{A}}^{\oplus (r+1)}\oplus \pi_*f^*N(\tilde{A})\to \pi_*f^*N(\tilde{A})|_{\tilde{A}}]
\]
supported in $[0,2]$. This is quasi-isomorphic to the complex of vector bundles
\[
[\pi_*L(\tilde{A})^{\oplus (r+1)}\oplus N^\lambda\to K]
\]
supported in $[0,1]$, where $K$ is the kernel of 
\begin{equation}
  \label{def of K}
  \pi_*L(\tilde{A})|_{\tilde{A}}^{\oplus (r+1)}\oplus \pi_*f^*N(\tilde{A})\to \pi_*f^*N(\tilde{A})|_{\tilde{A}}.
\end{equation}
\end{proof}

Let us write 
\begin{align*}
  \frF & \text{ for the vector bundle stack $h^1/h^0(R\pi_* f^* F)$} \\
  \frBlF & \text{ for the vector bundle stack $h^1/h^0(\BlF^\bullet)$} \\
  \frG^\lambda & \text{ for the vector bundle stack $h^1/h^0((G^\lambda)^{\bullet})$} && \lambda \in \{0\} \cup I \\
  \frBlG^\lambda & \text{ for the vector bundle stack $h^1/h^0(\BlGlambda^{\bullet})$} && \lambda \in \{0\} \cup I \\
\end{align*}
Dualizing the preceding discussion and applying the construction in \S\ref{deformations} gives deformations
\begin{align*}
  \frF\big|_{\tildeMX^0}  \rightsquigarrow N^0 \oplus \frG^0 &&
  \text{and} &&
  \frF\big|_{\tildeMXP^\lambda}  \rightsquigarrow N^\lambda \oplus \frG^\lambda. 
\end{align*}
Here we used the fact that, for $\lambda \in \{0\} \cup I$, the vector bundle stack $h^1/h^0((N^\lambda)^\bullet)$ is just the vector bundle $N^\lambda$.  Similarly there are deformations
\begin{align*}
  \frBlF\big|_{\BltildeMX^0} & \rightsquigarrow p_{\pp}^* N^0 \oplus \frBlG^0 
  \intertext{and} 
  \frBlF\big|_{\BltildeMXP^\lambda} & \rightsquigarrow p_{\pp}^* N^\lambda \oplus \frBlG^\lambda. 
\end{align*}

\section{The deformation space and its blow-up}

\begin{notation}
  Let $F$ and $G$ be Artin stacks.  Given a morphism $F \to G$ of Deligne--Mumford type, Kresch~\cite{Kresch} defines a stack $M^\circ_F G$ together with a flat morphism $M^\circ_F G \to \pp^1$ with general fibre $G$ and fibre over $0 \in \pp^1$ the normal cone $C_{F/G}$.  This is a generalisation of Fulton--MacPherson's \emph{deformation to the normal cone}~\cite{f}.  The construction is spelled out in detail in~\cite[Theorem~2.31]{eu}; see also~\cite{kkp}.  Since we already have a number of spaces called `$M$', we will use different notation, writing $\Def_F G$ for $M^\circ_F G$.
\end{notation}

\begin{remark} \label{section}
  There is a commutative diagram
  \[
    \xymatrix{
      F \times \pp^1 \ar[rd] \ar[r] & \Def_F G \ar[d] \\
      & \pp^1
    }
  \]
  where the diagonal arrow is projection to the second factor. 
\end{remark}

Let us now analyse the obstruction theory of $\BltildeMX \times \pp^1$ in $\Def_{\BltildeMP} \tildefrPic$.  Observe first that there are morphisms of complexes on $\tildeMX^0$ and $\tildeMXP^\lambda$
\begin{align} \label{G to L}
  (G^0)^{\bullet}\to R^\bullet \pi_* L^{\oplus (r+1)} && (G^\lambda)^{\bullet}\to R^\bullet \pi_* L^{\oplus (r+1)}
\end{align}
This follows by considering the morphism
\[\xymatrix{\pi_*L(\tilde{A})^{\oplus (r+1)}\oplus N^\lambda\ar[r]\ar[d]&\pi_*L(\tilde{A})^{\oplus (r+1)}\ar[d]\\
K\ar[r]&\pi_*L(\tilde{A})|_{\tilde{A}}^{\oplus (r+1)}
}
\]
where the upper horizontal arrow is the projection; the lower horizontal arrow is the embedding $K\to \pi_*L(\tilde{A})|_{\tilde{A}}^{\oplus (r+1)}\oplus \pi_*f^*N(\tilde{A})$ followed by projection to $\pi_*L(\tilde{A})|_{\tilde{A}}^{\oplus (r+1)}$; the vector bundle $K$ was defined in \eqref{def of K}; and $\lambda \in \{0\} \cup I$.  There are also canonical morphisms of complexes on $\tildeMX^0$ and $\tildeMXP^\lambda$
\begin{align} \label{F to L}
  R^\bullet \pi_* F & \to R^\bullet \pi_* L^{\oplus (r+1)} &   R^\bullet \pi_* F & \to R^\bullet \pi_* L^{\oplus (r+1)} 
\intertext{and}
  \label{F to G}   
  R^\bullet \pi_* F & \to  (G^0)^{\bullet} & R^\bullet \pi_* F &\to   (G^\lambda)^{\bullet} 
\end{align}
Here \eqref{F to L} arises from applying $R^\bullet \pi_*$ to the map $F \to L^{\oplus(r+1)}$ and \eqref{F to G} arises from the construction of $(G^\lambda)^\bullet$ as a cone, $\lambda \in \{0\} \cup I$.

Dualising \eqref{F to L} and applying Construction~\ref{con:deformation}  yields morphisms
\[
h:(R^\bullet\pi_*L^{\oplus (r+1)})^{\vee} \otimes O_{\pp^1}(-1)\to (R^\bullet\pi_*L^{\oplus (r+1)})^{\vee} \oplus (R^\bullet\pi_*F)^{\vee}
\]
over $\tildeMX^0 \times \pp^1$ and $\tildeMXP^\lambda \times \pp^1$, $\lambda \in I$.  Dualising \eqref{F to G} and applying Construction~\ref{con:deformation}  yields morphisms
\begin{align*}
  g: (G^0)^{\vee} \otimes O_{\pp^1}(-1)\to(G^0)^{\vee}\oplus (R^\bullet\pi_*F)^{\vee} &&& \text{over $\tildeMX^0 \times \pp^1$}\\
  g: (G^\lambda)^{\vee} \otimes O_{\pp^1}(-1)\to(G^\lambda)^{\vee}\oplus (R^\bullet\pi_*F)^{\vee} &&& \text{over $\tildeMXP^\lambda \times \pp^1$, $\lambda \in I$.}
\end{align*}
Write $c(g)$ for the mapping cone of $g$ and $c(h)$ for the mapping cone of $h$.  The morphism \eqref{G to L} induces a morphism of complexes from $c(g)$ to $c(h)$.

Consider the commutative diagram
\[
\xymatrix{
  (R^\bullet\pi_*L^{\oplus (r+1)})^{\vee} \ar[r] \ar[d] & (R^\bullet\pi_*F)^{\vee} \ar[d] \\ 
  j^* L^\bullet_{\tildeMP/\tildefrPic} \ar[r] & L^\bullet_{\tildeMX/\tildefrPic}
  }
\]
where $j$ is the inclusion of $\tildeMX$ in $\tildeMP$.  Applying Construction~\ref{con:deformation} to the bottom morphism yields 
\[
  l:  j^* L^\bullet_{\tildeMP/\tildefrPic}  \otimes O_{\pp^1}(-1) \to j^* L^\bullet_{\tildeMP/\tildefrPic} \oplus L^\bullet_{\tildeMX/\tildefrPic}
\]
over $\tildeMX \times \pp^1$, and applying the Four Lemma to the morphism
\[
  \xymatrix{
    j^* (E^\bullet_{\tildeMP/\tildefrPic})^\vee \otimes O_{\pp}(-1) \ar[r]^-{h}\ar[d] &
    j^*(E^\bullet_{\tildeMP/\tildefrPic})^\vee \oplus (R^\bullet \pi_* F)^\vee \ar[r]\ar[d] &
    c(h) \ar[d] \\
    j^* L^\bullet_{\tildeMP/\tildefrPic} \otimes O_{\pp^1}(-1) \ar[r]^-{l} & 
    j^* L^\bullet_{\tildeMP/\tildefrPic} \oplus L^\bullet_{\tildeMX/\tildefrPic} \ar[r] &
    c(l)
  }
\]
of distinguished triangles gives an embedding of cones 
\[
  h^1/h^0(c(l)^{\vee})\hookrightarrow h^1/h^0(c(h)^{\vee}).
\]
Writing $c(h)^{\geq -1}$ for the $[-1,0]$  truncation of the complex $c(h)$, we see that there is an embedding of cones
\[
  h^1/h^0(c(l)^{\vee})\hookrightarrow h^1/h^0((c(h)^{\geq -1})^{\vee}).
\]

\begin{lemma}
  The cone stack $h^1/h^0(c(h)^{\geq -1})^{\vee}$ over $\pp^1$ is isomorphic to 
\[
h^1/h^0(R^\bullet\pi_*f^*F)\times\A^1
\]
over $\pp^1\setminus 0$, and the fibre over $0 \in \pp^1$ is $\pi_*f^*N\oplus j^* \frE_{\tildeMP/\tildefrPic} $.
\end{lemma}
\begin{proof}
  Locally, we have that $c(h)$ is the complex
  \[
    \hspace{-3cm}
    \xymatrix{
      \left((R^\bullet\pi_*L(A)|_A^{\oplus (r+1)})^{\vee} \otimes O_{\pp^1}(-1) \right) 
      \oplus (\pi_*f^*N(A)|_A)^{\vee } \ar[d] \\ 
      \makebox[.8\textwidth][l]{ $
        \left((R^\bullet\pi_*L(A)^{\oplus (r+1)})^{\vee} \otimes O_{\pp^1}(-1) \right)  
        \oplus (R^\bullet\pi_*L(A)|_A^{\oplus (r+1)})^{\vee}
        \oplus (R^\bullet\pi_*L(A)|_A^{\oplus (r+1)})^{\vee}
        \oplus (\pi_* f^*N(A))^\vee $} \ar[d] \\
      (R^\bullet\pi_*L(A)^{\oplus (r+1)})^{\vee} \oplus (R^\bullet\pi_*L(A)^{\oplus (r+1)})^{\vee}
    }
  \]
  supported in $[-2,0]$.  The result follows.
\end{proof}

There is an entirely analogous story on the blown up moduli spaces.  There are morphisms of complexes
on $\BltildeMX^0$ and $\BltildeMXP^\lambda$
\begin{align} 
\label{Bl G to L}
  (\BlGzero)^{\bullet} & \to E^\bullet_{\BltildeMP/\tildefrPic} & (\BlGlambda)^{\bullet} & \to  E^\bullet_{\BltildeMP/\tildefrPic} \\
\intertext{as well as}
\label{Bl F to L}
  \BlF^\bullet & \to E^\bullet_{\BltildeMP/\tildefrPic}&   \BlF^\bullet & \to E^\bullet_{\BltildeMP/\tildefrPic}
\intertext{and}
\label{Bl F to G}   
  \BlF^\bullet & \to  (\BlGzero)^{\bullet} & \BlF^\bullet &\to   (\BlGlambda)^{\bullet} 
\end{align}
Here \eqref{Bl G to L} arises by arguing as in the proof of Lemma~\ref{lem:G}; \eqref{Bl F to L} 
arises by arguing as in the proof of Lemma~\ref{lem:N}; and
\eqref{Bl F to G} arises from the construction of $(G^\lambda)^\bullet$ as a cone, $\lambda \in \{0\} \cup I$.  Dualising \eqref{Bl F to L} and applying Construction~\ref{con:deformation}  yields morphisms
\[
Bl(h):(E^\bullet_{\BltildeMP/\tildefrPic})^\vee \otimes O_{\pp^1}(-1)\to (E^\bullet_{\BltildeMP/\tildefrPic})^{\vee} \oplus (\BlF^\bullet)^{\vee}
\]
over $\BltildeMX^0 \times \pp^1$ and $\BltildeMXP^\lambda \times \pp^1$, $\lambda \in I$.  Dualising \eqref{Bl F to G} and applying Construction~\ref{con:deformation}  yields morphisms
\[
  Bl(g): (\BlGzero)^{\bullet \vee} \otimes O_{\pp^1}(-1)\to(\BlGzero)^{\bullet \vee}\oplus (\BlF^\bullet)^{\vee} 
\]
over $\BltildeMX^0 \times \pp^1$ and
\[
  Bl(g): (\BlGlambda)^{\bullet \vee} \otimes O_{\pp^1}(-1)\to(\BlGlambda)^{\bullet \vee}\oplus (\BlF^\bullet)^{\vee}
\]
over $\BltildeMXP^\lambda \times \pp^1$, $\lambda \in I$.  The morphism \eqref{Bl G to L} induces a morphism of complexes between the mapping cones $c(Bl(g))$ and $c(Bl(h))$.  As before, applying Construction~\ref{con:deformation} to the morphism of cotangent complexes 
\[
j^* L_{\BltildeMP/\tildefrPic} \to L^\bullet_{\BltildeMX/\tildefrPic}
\]
yields
\[
  Bl(l):  j^* L^\bullet_{\BltildeMP/\tildefrPic}  \otimes O_{\pp^1}(-1) \to j^* L^\bullet_{\BltildeMP/\tildefrPic} \oplus L^\bullet_{\BltildeMX/\tildefrPic}
\]
over $\BltildeMX \times \pp^1$, and there is an embedding of cones 
\[
  h^1/h^0(c(Bl(l))^{\vee})\hookrightarrow h^1/h^0((c(Bl(h))^{\geq -1})^{\vee}).
\]

\begin{lemma} \label{Bl(h) family}
  The cone stack $h^1/h^0((c(Bl(h))^{\geq -1})^{\vee})$ over $\pp^1$ is isomorphic to 
\[
h^1/h^0(\BlF^\bullet)\times\A^1
\]
over $\pp^1\setminus 0$, and the fibre over $0 \in \pp^1$ is $\pi_*f^*N\oplus j^*\frE_{\BltildeMP/\tildefrPic}$.
\end{lemma}

\begin{proof}
  Locally, we have that $c(Bl(h))$ is the complex
  \[
    \hspace{-3cm}
    \xymatrix{
      \left(p_{\pp}^*(R^\bullet\pi_*L(A)|_A^{\oplus (r+1)})^{\vee} \otimes O_{\pp^1}(-1) \right) 
      \oplus p_{\pp}^*(\pi_*f^*N(A)|_A)^{\vee } \ar[d] \\ 
      \makebox[.8\textwidth][l]{ $
        \left(T_{\BltildeV/\tildefrPic}^{\vee} \otimes O_{\pp^1}(-1) \right)  
        \oplus p_{\pp}^*(R^\bullet\pi_*L(A)|_A^{\oplus (r+1)})^{\vee}
        \oplus p_{\pp}^*(R^\bullet\pi_*L(A)|_A^{\oplus (r+1)})^{\vee}
        \oplus p_{\pp}^*(\pi_* f^*N(A))^\vee $} \ar[d] \\
      T_{\BltildeV/\tildefrPic}^{\vee} \oplus T_{\BltildeV/\tildefrPic}^{\vee}
    }
  \]
supported in $[-2,0]$.  The result follows.
\end{proof}

\subsection{A blow-up of the deformation space}

The complex $c(Bl(h))$ fails to be perfect on a codimension-$2$ subset of $\Def_{\BltildeMP}\tildefrPic$.  We therefore consider the blow-up of the deformation space along this locus, and twist the analogous complex on the blow-up so that it becomes perfect.

\begin{construction} Consider the blow-up of $\Def_{\BltildeMP}\tildefrPic$ along the locus 
\[
  C_{\BltildeMP/\tildefrPic}\big|_{\BltildeMP^0\cap \BltildeMP^{\lambda}}
\]
in the fibre over $0 \in \pp^1$.  We denote the blown up space by 
\[
p:\Def^\prime_{\BltildeMP}\tildefrPic\to \Def_{\BltildeMP}\tildefrPic
\]
and the exceptional divisor by $D$.   Consider the Cartesian diagram
\begin{equation}
  \label{Z}
  \begin{aligned}
    \xymatrix{
      Z \ar[r] \ar[d]_{p_Z} & \Def^\prime_{\BltildeMP}\tildefrPic \ar[d]^{p} \\
      \BltildeMXP \times \pp^1 \ar[r] & \Def_{\BltildeMP}\tildefrPic
    }
  \end{aligned}
\end{equation}
where the bottom horizontal map arises from Remark~\ref{section}.  Let $D_Z$ denote the exceptional divisor for $p_Z$.
\end{construction}

Let us extend the Cartesian diagram \eqref{Z} to a larger Cartesian diagram
\begin{equation}
  \label{larger Cartesian diagram}
  \begin{aligned}
    \xymatrix{
      Z(X) \ar[r] \ar[d]_{p_{Z(X)}} & Z \ar[r] \ar[d]_{p_Z} & \Def^\prime_{\BltildeMP}\tildefrPic \ar[d]^{p} \\
      \BltildeMX \times \pp^1 \ar[r] & \BltildeMXP \times \pp^1 \ar[r] & \Def_{\BltildeMP}\tildefrPic.
    }
  \end{aligned}
\end{equation}

\begin{lemma} \label{Bl(tildeh)}
  Let $c(\tilde{h})$ denote the complex $\big(p_Z^* c(Bl(h))\big)^{\geq -1}$.   Then there are embeddings of cones 
  \[
    C_{Z(X)/\Def^\prime_{\BltildeMP}\tildefrPic} \hookrightarrow 
    h^1/h^0(c(Bl(l))^\vee) \hookrightarrow 
    h^1/h^0(c(\tilde{h})^\vee) 
  \]
  where $c(\tilde{h})$ is the mapping cone of $\tilde{h}$.
\end{lemma}
\begin{proof}  
  Recall the definition of $Bl(l)$ just above Lemma~\ref{Bl(h) family}.   Kim--Kresch--Pantev prove that there is an embedding of cones~\cite[Proposition~1]{kkp}
  \[
    C_{Z(X)/\Def^\prime_{\BltildeMP}\tildefrPic} \hookrightarrow h^1/h^0(c(Bl(l))^\vee).
  \]
  The discussion just before Lemma~\ref{Bl(h) family} shows that $h^1/h^0((c(Bl(h))^{\geq -1})^{\vee})$ contains $h^1/h^0(c(Bl(l))^{\vee})$, and we have that $\big(p_Z^* c(Bl(h))\big)^{\geq -1}=c(\tilde{h})$.
\end{proof}

\begin{lemma} \label{normal sheaf} $h^1/h^0(c(Bl(g))^\vee)$ contains the abelian cone stack 
\[
N_{\tildeMX\times\pp^1/\Def_{\tildeMP}\tildefrPic}|_{\BltildeMX^{\lambda}}
\]
associated to the normal sheaf of $\tildeMX\times\pp^1$ in $\Def_{\tildeMP}\tildefrPic$.
\end{lemma}
\begin{proof}
By \cite[Proposition~1]{kkp} we have that $N_{\tildeMX\times\pp^1/\Def_{\tildeMP}\tildefrPic}\simeq h^1/h^0(c(Bl(l))^{\vee})$. As before, there is a morphism of distinguished triangles
\[
\xymatrix{ \BlGlambda \otimes O_{\pp^1}(-1)\ar[r]^-{Bl(g)}\ar[d]& \BlGlambda\oplus (\BlF)^{\vee}\ar[r]\ar[d]&c(Bl(g))\ar[d]\\
j^*L_{\BltildeMP/\tildefrPic} \otimes O_{\pp^1}(-1)\ar[r]^-{Bl(l)}& j^* L_{\BltildeMP/\tildefrPic}\oplus L_{\BltildeMX/\tildefrPic}\ar[r] &c(Bl(l))
}
\]
over $\BltildeMX \times \pp^1$.  The Four Lemma implies the conclusion.
\end{proof}
\section{The main theorem}
\label{sec:main theorem}

Let $K^{\lambda}$, $\lambda \in I$, denote the vector bundle on $\BltildeMXP^\lambda$ given by the kernel of $p_{\pp}^* N^{\lambda}\to p_{\pp}^* \pi_*f^*N|_{Q_\lambda}$, where $Q_\lambda$ is a node on the contracted elliptic component $E$ that separates $E$ from a rational component $R$ of the curve.  Define:
\begin{align*}
  [\BltildeMX^0]^{\vv} &=  0^!_{p_X^*N^0}\left[C_{\BltildeMX^0/[\BltildeMP^0]^\vv} \right] \\
  [\BltildeMX^{\lambda}]^{\vv} &=  0^!_{K^\lambda}\left[C_{\BltildeMX^{\lambda}/[\BltildeMXP^{\lambda}]^{\vv}}\right] \\
  [\tildeMX^\lambda]^\vv &= p_{X*} [\BltildeMX^{\lambda}]^{\vv}
\end{align*}
We are now in a position to state and prove our main result.

\begin{theorem}  \label{main theorem} There is an equality
  \[ 
    [\BltildeMX]^{\vv}=[\BltildeMX^0]^{\vv} + \sum_{\lambda \in I} [\BltildeMX^{\lambda}]^{\vv} 
  \]
  in $A_*(\BltildeMX)$.
\end{theorem}

\noindent This implies the promised decomposition of the virtual class on $\tildeMX$.

\begin{corollary} \label{splitting corollary} There is an equality
  \[ 
    [\tildeMX]^{\vv}=[\tildeMX^0]^{\vv} + \sum_{\lambda \in I} [\tildeMX^{\lambda}]^{\vv} 
  \]
  in $A_*(\tildeMX)$.
\end{corollary}
\begin{proof}[Proof of Corollary~\ref{splitting corollary}]
 Combine Theorem~\ref{main theorem}, Lemma~\ref{bl to tilde}, and the definition of the virtual class on $\tildeMX^{\lambda}$.
\end{proof}

\begin{remark}
  We could also define $[\tildeMX^0]^\vv = p_{X*} [\BltildeMX^0]^{\vv}$, by analogy with the definition of $[\tildeMX^\lambda]^\vv$.  Since pullback commutes with pushforward in the Cartesian diagram
  \[
    \xymatrix{
      \BltildeMX^0 \ar[r] \ar[d]_{p_X} &  \BltildeMP^0 \ar[d]^{p_{\pp}}  \\
      \tildeMX^0 \ar[r]  &  \tildeMP^0 
    }
  \] 
  we see that this agrees with the Vakil--Zinger definition 
  \[
    [\tildeMX^0]^{\vv}=0_{N^0}^![C_{\tildeMX^0/\tildeMP^0}].
  \]
\end{remark}

\subsection{Proof of Theorem~\ref{main theorem}}

Consider now
\[
\DD' := \Def_{Z(X)} \Def^\prime_{\BltildeMP}\tildefrPic.
\]
By construction we have a morphism $\DD^\prime \to \pp^1 \times \pp^1$.  Restricting to zero we get a morphism $C_{Z(X)/\Def^\prime_{\BltildeMP}\tildefrPic}\to \pp^1$. This morphism may not be flat; it has general fiber $C_{\BltildeMX/\tildefrPic}$. Restricting in the other direction we get a flat morphism 
\[
\Def_{Z(X)_0} \Def^\prime_0 \to\pp^1
\]
where $\Def^\prime_0$ and $Z(X)_0$ are the fibres over $0 \in \pp^1$ of, respectively, $\Def^\prime_{\BltildeMP}\tildefrPic$ and $Z(X)$ in the following diagram:
\[
\xymatrix{
  Z(X) \ar[rr]^-{p_{Z(X)}} \ar[d] && \BltildeMX \times \pp^1 \ar[d] \\
  \Def^\prime_{\BltildeMP}\tildefrPic \ar[rr]^-p \ar[rd] &&
  \Def_{\BltildeMP}\tildefrPic \ar[ld] \\
  & \pp^1
}
\]
The fibre $\Def^\prime_0$ is the union of $C_{\BltildeMP/\tildefrPic}$ and the exceptional divisor $D$ for~$p$.  Commutativity of intersection with divisors gives that 
\[
[C_{\BltildeMX/\tildefrPic}]=[C_{Z(X)_0/\Def^\prime_0}]
\]
in $A_*(\DD^\prime)$.

We now write the fibre $\Def^\prime_0$ as a union of components.  As $\BltildeMP$ is a union of components 
\[
\BltildeMP^0\cup \bigcup_{\lambda \in I} \BltildeMP^{\lambda}
\]
the cone $C_{\BltildeMP/\tildefrPic}$ is also a union of components, which are supported on the components of $\BltildeMP$. We write 
\[
C_{\BltildeMP/\tildefrPic}=C_{\BltildeMP/\tildefrPic}^0\cup \bigcup_{\lambda \in I} C_{\BltildeMP/\tildefrPic}^{\lambda}.
\]
This splitting is unique as $C_{\BltildeMP/\tildefrPic}$ does not have components supported at intersections of components of $\BltildeMP$: this is clear from the Hu--Li local equations~\cite[Theorem~2.19]{huli}; cf.~\cite[Example~3.4]{bcm}.  Thus
\[
\Def^\prime_0 = C_{\BltildeMP/\tildefrPic}^0\cup \bigcup_{\lambda \in I} C_{\BltildeMP/\tildefrPic}^{\lambda} \cup D.
\]

This decomposition induces a decomposition of the cone $C_{Z(X)_0/\Def^\prime_0}$ as a union of components.  Write $D^\lambda$ for the union of components of $D$ that lie over $\BltildeMX^\lambda$. Consider the fiber product
\[
  D^\lambda_{Z(X)} := Z_X \times_{\Def^\prime_{\BltildeMP}\tildefrPic} D^\lambda.
\]
This sits in a Cartesian diagram
\begin{equation}
  \label{largest Cartesian diagram}
  \begin{aligned}
    \xymatrix{
      D^\lambda_{Z(X)} \ar[d]_{i_{Z(X)}} \ar[r] & D^\lambda_Z \ar[d] \ar[r] & D^\lambda \ar[d] \\
      Z(X) \ar[r] \ar[d]_{p_{Z(X)}} & Z \ar[r] \ar[d]_{p_Z} & \Def^\prime_{\BltildeMP}\tildefrPic \ar[d]^{p} \\
      \BltildeMX \times \pp^1 \ar[r] & \BltildeMXP \times \pp^1 \ar[r] & \Def_{\BltildeMP}\tildefrPic.
    }
  \end{aligned}
\end{equation}
where the lower part is diagram \eqref{larger Cartesian diagram}.  Note that if we replaced $D^\lambda$ in this diagram by another component $C_{\BltildeMP/\tildefrPic}^{\lambda}$ of $\Def^\prime_0$, $\lambda \in \{0\} \cup I$, then the corresponding fiber product would just be $\BltildeMX^\lambda$.  Thus
\[
  C_{Z(X)_0/\Def^\prime_0} = 
  C_{\BltildeMX^0/C^0_{\BltildeMP/\tildefrPic}} \cup
  \bigcup_{\lambda\in I} C_{\BltildeMX^\lambda/C^{\lambda}_{\BltildeMP/\tildefrPic}} \cup
  \bigcup_{\lambda\in I} C_{D^\lambda_{Z(X)}/D^\lambda}.
\]

Each of these components of $C_{Z(X)_0/\Def^\prime_0}$ embeds into a vector bundle stack.
Recall that
\begin{align*}
  C_{\BltildeMX^0/C^0_{\BltildeMP/\tildefrPic}} \hookrightarrow C_{\BltildeMX^0/\BltildeMP^0}\times C^0_{\BltildeMP/\tildefrPic}|_{\BltildeMX}
  & \hookrightarrow N^0\oplus \frE^0\\
  \intertext{and}
  C_{\BltildeMX^\lambda/C^{\lambda}_{\BltildeMP/\tildefrPic}} \hookrightarrow C_{\BltildeMX^{\lambda}/\BltildeMP^{\lambda}}\times C^{\lambda}_{\BltildeMP/\tildefrPic}|_{\BltildeMX} 
  & \hookrightarrow N^{\lambda}\oplus \frE^{X,\lambda};
\end{align*}
here we used Lemma~\ref{reduced obstruction}.  Furthermore
\[
  C_{D^\lambda_{Z(X)}/D^\lambda} \hookrightarrow i_{Z(X)}^* \, p_{Z(X)}^* \left( N^{\lambda}\oplus \frE^{X,\lambda} \right), 
\]
since
\begin{align*}
  C_{D^\lambda_{Z(X)}/D^\lambda} & \hookrightarrow i_{Z(X)}^* \, p_{Z(X)}^* \left(C_{\BltildeMX \times \pp^1/\Def_{\BltildeMP} \tildefrPic} \right) && \text{by diagram \eqref{largest Cartesian diagram}} \\
  & \hookrightarrow i_{Z(X)}^* \, p_{Z(X)}^* \left(N_{\BltildeMX \times \pp^1/\Def_{\BltildeMP} \tildefrPic} \right) && \text{by definition} \\
  & \hookrightarrow i_{Z(X)}^* \, p_{Z(X)}^* h^1/h^0(c(Bl(g))^\vee) && \text{by Lemma~\ref{normal sheaf}}
\end{align*}
and the fiber of $h^1/h^0(c(Bl(g))^\vee)$ over $\BltildeMX^\lambda \times \{0\}$ is $N^{\lambda}\oplus \frE^{X,\lambda}$.  

We now look at embeddings of families of cones in vector bundles. For this we write $C_X = C_{\BltildeMX/\tildefrPic}$ as a union of components $C_X^0\cup \bigcup_{\lambda \in I} C_X^{\lambda}$ such that $C_X^0$ is supported on the main component $\BltildeMX^0$, $C_X^{\lambda}$ is supported on the ghost component $\BltildeMX^{\lambda}$, and $[C_X]=[C_X^0]+\sum_{\lambda\in I}[C_X^{\lambda}]$ in $A_*(C_X)$. Such expressions are not unique, as $C_X$ may have components supported on $\BltildeMX^0\cap \BltildeMX^{\lambda}$. After making such a choice we have
\begin{multline*}
  [C_X^0]+\sum_{\lambda \in I} [C_X^{\lambda}]= \\
  [C_{\BltildeMX^0/C^0_{\BltildeMP/\tildefrPic}}]+
  \sum_{\lambda \in I}[C_{\BltildeMX^{\lambda}/C^{\lambda}_{\BltildeMP/\tildefrPic}}]+ 
  \sum_{\lambda \in I} [C_{D^\lambda_{Z(X)}/D^\lambda}]
\end{multline*}
in $A_*(\DD^\prime)$. Suppose that $C_X^0$ deforms to $C^0_{\limiting}$ in $\DD^\prime$, and that $C^{\lambda}$ deforms to $C^{\lambda}_{\limiting}$ in $\DD^\prime$; here $C^{0}_{\limiting}$ and $C^{\lambda}_{\limiting}$ can be unions of components. Then
\begin{multline}
  \label{eq:support} [C^0_{\limiting}]-[C_{\BltildeMX^0/C^0_{\BltildeMP/\tildefrPic}}]=\\
  \sum_{\lambda \in I} [C_{\BltildeMX^{\lambda}/C^{\lambda}_{\BltildeMP/\tildefrPic}}] -
  \sum_{\lambda \in I} [C^{\lambda}_{\limiting}] + 
  \sum_{\lambda \in I} [C_{D^\lambda_{Z(X)}/D^\lambda}]
\end{multline}
in $A_*(\DD^\prime)$. Denote 
\[
 [C_{\limiting}^0]-[C_{\BltildeMX^0/C^0_{\BltildeMP/\tildefrPic}}]
\]
by $[\Corr]$ and note that by \eqref{eq:support} we may assume that $[\Corr]$ is supported on the intersection of the main component with the ghost components. Note that $\Corr$ may not be an effective cycle. 

Consider the exact sequence 
\[
\pi_*f^*N(\tilde{A})\to \pi_*f^*N(\tilde{A})|_{\tilde{A}}\to R^1\pi_*f^*N \to 0
\]
on $\BltildeMX \times \pp^1$, and pull it back to $Z$ to obtain an exact sequence
\begin{equation}\label{seq: tilde a}
p_Z^*\pi_*f^*N(\tilde{A})\to p_Z^*\pi_*f^*N(\tilde{A})|_{\tilde{A}}\to p_Z^*(R^1\pi_*f^*N) \to 0.
\end{equation}
Denote the kernel of the left-hand map by $N'$. In the following we show that $N'$ is a vector bundle on $\BltildeMX^0$. Restricting to the main component, we see that $p_Z^*(R^1\pi_*f^*N)$ is supported on the divisor $D_Z$.  Locally, the complex 
\[
[p_Z^*\pi_*f^*N(\tilde{A})\to p_Z^*\pi_*f^*N(\tilde{A})|_{\tilde{A}}],
\]
 is quasi-isomorphic to
\[
[p_Z^*\pi_*f^*N(A)\to p_Z^*\pi_*f^*N(A)|_A]
\]
and the map factors as 
\[
p_Z^*\pi_*f^*N(A)\to p_Z^*\pi_*f^*N(A)|_A \otimes O(-D_Z) \to p_Z^*\pi_*f^*N(A)|_A
\]
where the right-hand map is multiplication by $D_Z$ and the left-hand map is surjective. Since locally $N'$ is the kernel of a surjective map to a line bundle, it is a vector bundle on $\BltildeMX^0$.

The fiber of $h^1/h^0(c(\tilde{h})^\vee)$ over $0 \in \pp^1$ is $N^{\prime}\oplus  p_Z^*\frE^0$.  Lemma~\ref{Bl(tildeh)} therefore implies that $C_{\limiting}^0\hookrightarrow N^{\prime}\oplus  p_Z^*\frE^0$. This gives a class 
\[
[\Corr]^{\vv}:=0^!_{N^{\prime}\oplus  p_Z^*\frE^0} [\Corr]
\]
and we get
\begin{align*}
  [\Corr]^\vv &= 
  0^!_{N^{\prime}\oplus  p_Z^*\frE^0} \left[C^0_{\limiting}\right] - 0^!_{N^{\prime}\oplus  p_Z^*\frE^0} \left[C_{\BltildeMX^0/C^0_{\BltildeMP/\tildefrPic}}\right] \\
  &= 0^!_{h^1/h^0(E_{\BltildeMX/\tildefrPic})} \left[C^0_X\right]  - 0^!_{N^{\prime}} \left[C_{\BltildeMX^0/[\BltildeMP^0]^\vv} \right] 
\intertext{by deformation invariance (for the first term) and the definition of the virtual class on $\BltildeMP^0$ (for the second term)}
  &= 0^!_{h^1/h^0(E_{\BltildeMX/\tildefrPic})} \left[C^0_X\right] - 0^!_{p_{\pp}^* N^0} \left[C_{\BltildeMX^0/[\BltildeMP^0]^\vv} \right] \\ 
\intertext{because $N'$ restricts to $p_{\pp}^* N^0$ along $\BltildeMX^0$}
  &= 0^!_{h^1/h^0(E_{\BltildeMX/\tildefrPic})} \left[C^0_X\right] - \left[\BltildeMX^0\right]^\vv.
\end{align*}
Thus $C_X^0$ contributes to $[\BltildeMX]^{\vv}$ the class $[\BltildeMX^0]^{\vv}+[\Corr]^{\vv}$.

Now fix $\lambda \in I$.  The vector bundle stack $p_{Z}^* \, h^1/h^0(c(Bl(g))^\vee)$ on $Z$ has fiber over $0 \in \pp^1$ equal to $N^{\lambda}\oplus p_Z^* \frBlG^{\lambda}$ and general fiber equal to $p_Z^* \, h^1/h^0(\BlF^{\bullet})$.  Lemma~\ref{normal sheaf} thus implies that $C^\lambda_{\limiting} \hookrightarrow N^{\lambda}\oplus p_Z^* \frBlG^{\lambda}$, and we have
\[
\left[C^\lambda_{\limiting}\right] = \left[C^\lambda_X\right]
\]
 in $A_*(\DD^\prime)$.  From \eqref{eq:support} again we have that
\[
  \left[ \Corr \right] = 
  \sum_{\lambda \in I} \left[C_{\BltildeMX^{\lambda}/C^{\lambda}_{\BltildeMP/\tildefrPic}}\right] -
  \sum_{\lambda \in I} \left[C^{\lambda}_{\limiting}\right] + 
  \sum_{\lambda \in I} \left[C_{D^\lambda_{Z(X)}/D^\lambda}\right]
\]
and Lemma~\ref{two bundles} implies that
\begin{multline}
  \label{corr again}
  \left[ \Corr \right]^\vv  = 
  \sum_{\lambda \in I} 0^!_{N^{\lambda}\oplus p_Z^* \frBlG^{\lambda}}\left[C_{\BltildeMX^{\lambda}/C^{\lambda}_{\BltildeMP/\tildefrPic}}\right] 
  \\ 
   - \sum_{\lambda \in I} 0^!_{N^{\lambda}\oplus p_Z^* \frBlG^{\lambda}}\left[C^{\lambda}_{\limiting}\right] + 
  \sum_{\lambda \in I} 0^!_{N^{\lambda}\oplus p_Z^* \frBlG^{\lambda}}\left[C_{D^\lambda_{Z(X)}/D^\lambda}\right]
\end{multline}
The first summand on the right-hand side of \eqref{corr again} is
\begin{equation}
  \label{first summand}
  0^!_{K^{\lambda}\oplus p_Z^* \frE^{X, \lambda}}\left[C_{\BltildeMX^{\lambda}/C^{\lambda}_{\BltildeMXP/\tildefrPic}}\right] 
\end{equation}
where $K^\lambda$ was defined just above Theorem~\ref{main theorem}.  Here we used the fact that the difference between $C^{\lambda}_{\BltildeMXP/\tildefrPic}$ and $C^{\lambda}_{\BltildeMP/\tildefrPic}$, which is a vector bundle, coincides with the difference between $p_Z^* \frE^{\lambda}$ and $p_Z^* \frE^{X,\lambda}$; note that
 $p_Z^* \frE^{X,\lambda}$ and $p_Z^* \frBlG^{\lambda}$ coincide.  The local model here is~\cite[Example~3.12(a)]{bcm}.  Arguing as in the $\lambda=0$ case, \eqref{first summand} is
 \[
0^!_{K^{\lambda}}\left[C_{\BltildeMX^{\lambda}/[\BltildeMXP^\lambda]^\vv}\right].
\]
Furthermore, Lemma \ref{D contribution} implies that the third sum in \eqref{corr again} vanishes.  It follows that $\bigcup_{\lambda \in I} C_X^{\lambda}$ contributes to $[\BltildeMX]^{\vv}$ the class 
\[
\sum_{\lambda \in I} [\BltildeMX^{\lambda}]^{\vv} -[\Corr]^{\vv}.
\]
Adding this to the contribution from $C_X^0$ proves Theorem~\ref{main theorem}. 

\begin{lemma}\label{two bundles} We have that
\[
[ \Corr]^\vv =  0^!_{N^{\lambda}\oplus p_Z^* \frBlG^{\lambda}}[\Corr].
  \]
\end{lemma}

\begin{proof} 
By definition we have 
\[
[ \Corr]^\vv =0^!_{N'\oplus  p_Z^*\frE^0} [\Corr].
\]
Since $\Corr$ embeds in $N'\oplus p_Z^*C^0$ and $p_Z^*C^0$ embeds in $p_Z^*\frE^0$, there is  a factorisation
\[
\Corr \hookrightarrow N'\oplus p_Z^*C^0\hookrightarrow N'\oplus p_Z^*\frE^0.
\] 
From the Cartesian diagram
\[\xymatrix {\Corr \cap (C^0\cup D_Z)\ar[r]\ar[d]&N'\ar[d]\ar[r]&N'\ar[d]\\
\Corr\ar[r]& N'\oplus C^0\cup D_Z\ar[r]&N'\oplus p_Z^*\frE^0
}
\]
we get that 
\[
[\Corr]^{\vv}=0^!_{N'}0^!_{C^0\cup D_Z}\left([\Corr]\cdot  \ctop (p_Z^*\Cok^0)\right)
\]
where $\Cok^0$ was defined in Remark~\ref{define Cok}.  

We now compute $0^!_{N^{\lambda}\oplus p_Z^* \frBlG^{\lambda}}[\Corr]$. There is a Cartesian diagram
\[\xymatrix {\Corr\ \cap (C^{\lambda}\cup D_Z)\ar[r]\ar[d]&N^{\lambda}\ar[d]\ar[r]&N^{\lambda}\ar[d]\\
\Corr\ar[r]& N^{\lambda}\oplus C^{\lambda}\cup D_Z\ar[r]&N^{\lambda}\oplus p_Z^*\frE^{X,\lambda}
}
\]
and so
\begin{align*}0^!_{N^{\lambda}\oplus p_Z^* \frE^{X,\lambda}}[\Corr] &=
0^!_{N^{\lambda}}0^!_{C^{\lambda}\cup D_Z}\left([\Corr]\cdot  \ctop (p_Z^*\Cok^{\lambda})\right)\\
&=0^!_{N'}0^!_{C^{\lambda}\cup D_Z}\left([\Corr]\cdot  \ctop (p_Z^*\Cok^{\lambda})\cdot \ctop(p_Z^*R^1\pi_*f^*N(-D_Z))\right)\\
&=0^!_{N'}0^!_{C^{\lambda}\cup D_Z}\left([\Corr]\cdot  \ctop (p_Z^*\Cok^{0})\right). 
\end{align*}
On the intersection of the main component with the ghost components we have that $C^0\simeq C^{\lambda}$. The result follows.
\end{proof}

\begin{lemma}\label{D contribution} For any $\lambda \in I$ we have that
\[   (p_Z)_* \left(0^!_{p_Z^*N^{\lambda}\oplus p_Z^* \frBlG^{\lambda}} \left[C_{D^\lambda_{Z(X)}/D^\lambda}\right]\right)=0.
\]
\end{lemma}
\begin{proof}
Define $[D^{\lambda}_Z]^{\vv}=\ctop(p_Z^*\Cok^{\lambda})\cdot [D^{\lambda}_Z]$, where $\Cok^\lambda$ is as in Remark~\ref{define Cok}. Then by functoriality of virtual pull backs~\cite{eu}, we have that
\[0^!_{p_Z^*N^{\lambda}\oplus p_Z^* \frBlG^{\lambda}} \left[C_{D^\lambda_{Z(X)}/D^\lambda}\right]=
0^!_{p_Z^*N^{\lambda}} \left[C_{D^\lambda_{Z(X)}/[D^\lambda]^{\vv}}\right]
\]
By commutativity of pull backs with push forwards we get that
\[(p_Z)_* \left(0^!_{p_Z^*N^{\lambda}} \left[C_{D^\lambda_{Z(X)}/[D^\lambda]^{\vv}}\right]\right)= \left(0^!_{N^{\lambda}} \left[C_{\tildeMX^0\cap \tildeMX^{\lambda}/(p_Z)_*[D^\lambda]^{\vv}}\right]\right).
\]
Since the virtual dimension of $D^{\lambda}$ is equal to the dimension of $\tildeMP$ and the dimension of $p_Z(D^\lambda)$ is strictly smaller than the dimension of $\tildeMP$ it follows that $(p_Z)_*[D^\lambda]^{\vv}=0$. This proves the Lemma.
\end{proof}

\section{Contributions from ghost components}\label{contribution}
We show that the splitting in \S\ref{sec:main theorem} is compatible with push forwards.
\subsection{Virtual push forwards} 
Consider
\begin{align*}
  M(X)^{\lambda} & =\bar{M}_{1,n_0+k}\times_{X} \bar{M}_{0,n_1+1}(X,d_1)\times_X \cdots \times_X \bar{M}_{0,n_k+1}(X,d_k)  \, \big/ \, \Gamma^\lambda\\
  P(X)^{\lambda} &= \bar{M}_{0,n_1+1}(X,d_1)\times_{X} \cdots \times_{X} \bar{M}_{0,n_k+1}(X,d_k)  \, \big/ \, \Gamma^\lambda \\
  \intertext{and}
  P(\pp)^{\lambda} &= \bar{M}_{0,n_1+1}(\pp,d_1)\times_{\pp} \cdots \times_{\pp} \bar{M}_{0,n_k+1}(\pp,d_k)  \, \big/ \, \Gamma^\lambda.
\end{align*}
Here $\lambda$ denotes the combinatorial data $(k; n_0,\ldots,n_k; d_1,\ldots, d_k)$ and $\Gamma^\lambda$ is the (finite) automorphism group of this data. See \cite{leeoh} for more details.
$P(\pp)^{\lambda}$ is smooth and we take $[P(\pp)^{\lambda}]^{\vv}=[P(\pp)^{\lambda}]$. Define a virtual class on $P(X)^{\lambda}$ by
\[[P(X)^{\lambda}]^{\vv}= [\bar{M}_{0,n_1+1}(X,d_1)]^{\vv}\times_X...\times_X [\bar{M}_{0,n_k+1}(X,\beta_k)]^{\vv}  \, \big/ \, |\Gamma^\lambda|.
\]
Recall that the ghost components are indexed by $\lambda \in I$.  Let $J \subset I$ denote the set of indices of ghost components consisting of maps from curves such that the irreducible genus one component carries no marked points.   Let $1 \in J$ denote the index of the ghost component whose generic point consists of a collapsed genus one component with no marked points, attached to a single $n$-marked, degree-$d$ rational tail.  Let $M_{0}(X)$ denote $\bar{M}_{0,n}(X, d)$, and let 
\begin{align*}q^1:\tildeMX^1 &\to M_{0}(X) &&&
q^{\lambda}: \tildeMX^{\lambda} &\to P(X)^{\lambda}
\end{align*} 
denote the natural projections. We are interested in computing 
\[
(q^{\lambda})_*\ev^*\gamma\cdot [\tildeMX^{\lambda}]^{\vv}.
\]
 The projection formula implies that 
\begin{equation}\label{proj}(q^{\lambda})_*\ev^*\gamma\cdot [\tildeMX^{\lambda}]^{\vv}=\ev^*\gamma\cdot q^{\lambda}_* \, [\tildeMX^{\lambda}]^{\vv}.
\end{equation} 
To compute $(q^{\lambda})_* \, [\tildeMX^{\lambda}]^{\vv}$, we will show that $q^{\lambda}$ satisfies the virtual push-forward property \cite{eu2} whenever $\lambda \in J$. 

\begin{lemma}\label{xzcompat} There is a Cartesian diagram
\[\xymatrix{\BltildeMX^1\ar[r]\ar[d]\ar@/_2pc/[dd]_(.30){\tilde{q}^1}&\BltildeMXP^1\ar[d]\ar[r]\ar@/^2pc/[dd]^(.30){\tilde{r}^1}&\BltildeMP^1\ar[d]\\
\tildeMX^1\ar[r]\ar[d]^{q^1}&\tildeMXP^1\ar[d]\ar[r]&\tildeMP^1\ar[d]\\
M_{0,n+1}(X)\ar[r]^i&\tilde{r}^1(\tildeMXP^1)\ar[r]&M_{0,n+1}(\pp)}
\]
and $\tilde{r}^1(\tildeMXP^1)=:M^X_{0,n+1}(\pp)$ has a perfect dual obstruction theory $E^{\bullet}_{M^X_{0,n+1}(\pp)/M_{0,n+1}(\pp)}$ such that  
\[
\left(E^{\bullet}_{M_{0,n+1}(X)/M^X_{0,n+1}(\pp)},E^{\bullet}_{M_{0,n+1}(X)/M_{0,n+1}(\pp)}, E^{\bullet}_{M^X_{0,n+1}(\pp)/M_{0,n+1}(\pp)}\right)
\]
 is a compatible triple of dual obstruction theories.
\end{lemma}
\begin{proof} 
That the lower left and upper squares in the diagram are Cartesian follows from the definitions of $\tildeMXP$, $\BltildeMX$, $\BltildeMXP^1$, and $\BltildeMP^1$. To find a perfect dual obstruction theory for $M^X_{0,n+1}(\pp)$, we note that $M^X_{0,n+1}(\pp)=\ev_{n+1}^{-1}X\cap M_{0,n+1}(\pp)$. This shows that 
\[
E^{\bullet}_{M^X_{0,n+1}(\pp)/M_{0,n+1}(\pp)}= [0 \to \pi_*f^*N|_Q]
\] 
where the complex on the right is concentrated in $[0,1]$. Let $K_0^1$ be defined by the following exact sequence.
\begin{equation}\label{compat Y}
0\to K^1_0\to\pi^0_*f^*N\to \pi^0_*f^*N|_Q\to 0,
\end{equation}
where $\pi^0:M_{0,n+2}(\pp)\to M_{0,n+1}(\pp)$ is the universal curve\footnote{Note that $K^1$ in \S\ref{sec:main theorem} is $(\tilde{q}^1)^*K^1_0$.}. Define
\[E^{\bullet}_{M_{0,n+1}(X)/M^X_{0,n+1}(\pp)}=[0 \to K^1_0].
\]
By the Four Lemma, $E^{\bullet}_{M_{0,n+1}(X)/M^X_{0,n+1}(\pp)}$ is an obstruction theory. The compatibility of the triple 
\[
\left(E^{\bullet}_{M_{0,n+1}(X)/M^X_{0,n+1}(\pp)},E^{\bullet}_{M_{0,n+1}(X)/M_{0,n+1}(\pp)}, E^{\bullet}_{M^X_{0,n+1}(\pp)/M_{0,n+1}(\pp)}\right)
\] 
is equivalent to the exactness of (\ref{compat Y}).
\end{proof}

\begin{proposition} \label{one}$E_{\BltildeMX^1/\BltildeMXP^1}\simeq (\tilde{q}^1)^*E_{M_{0,n+1}(X)/M^X_{0,n+1}(\pp)}$ and \[\tilde{q}^1_*\,[\BltildeMX^1]^{\vv}=i^*D\cdot [M_{0,n+1}(X)]^{\vv},\]
with $D$ a divisor on $M^X_{0,n+1}(\pp)$.
\end{proposition}

\begin{proof}  The isomorphism follows from the Cartesian diagram
\begin{equation}\label{vpfdelta}
\xymatrix{\BltildeMX^1\ar[r]\ar[d]_{\tilde{q}^1}& \BltildeMXP^1\ar[d]^{\tilde{r}^1}\\
M_{0,n+1}(X)\ar[r]^{j_0} &M^X_{0,n+1}(\pp)
}
\end{equation}
Let $K^1$ be as in \S\ref{sec:main theorem}. By the above lemma it suffices to show that 
\[\tilde{q}^1_*i^!_{K^1}[\BltildeMXP]^{\vv}=i^*D\cdot [M_{0,n+1}(X)]^{\vv},
\]
with $D$ a divisor on $M^X_{0,n+1}(\pp)$. Since pull-backs commute with push forwards we have
\[
\tilde{q}^1_* i^!_{K^1} [\BltildeMXP]^{\vv}=i^!_{K^1}\tilde{r}^1_*[\BltildeMXP]^{\vv}
\]
Since $M_{0,n+1}^X(\pp)$ is smooth, we have that 
\[\tilde{r}^1_*[\BltildeMXP]^{\vv}=i^!_{K^1}D\cdot[M_{0,n+1}^X(\pp)]
\] for some divisor $D$ in $M_{0,n+1}^X(\pp)$. We thus get
\[
\tilde{q}^1_*i^!_{K^1}[\BltildeMXP]^{\vv}=i^!_{K^1}D\cdot[M_{0,n+1}^X(\pp)] =
i^*D\cdot [M_{0,n+1}(X)]^{\vv}.
\]
\end{proof}

\begin{proposition}\label{lambda}  $E_{\BltildeMX^{\lambda}/\BltildeMP^{\lambda}}\simeq (\tilde{q}^{\lambda})^*E_{M_{0,1}(X)/M_{0,1}(\pp)}$ and, for $\lambda \in J$ with $\lambda \ne 1$:
\[
\tilde{q}^{\lambda}_* \,[\BltildeMX^{\lambda}]^{\vv}=0.
\]
\end{proposition}

\begin{proof}  The isomorphism follows from the Cartesian diagram
\[
\xymatrix{\BltildeMX^{\lambda}\ar[r]\ar[d]_{\tilde{q}^{\lambda}}&\BltildeMXP^{\lambda}\ar[d]^{\tilde{r}^{\lambda}}\ar[r]&\BltildeMP^{\lambda}\ar[d]\\
P(X)^{\lambda}\ar[r]&\tilde{r}^{\lambda}(\tildeMXP)\ar[r]&P(\pp)^{\lambda}}
\]
By functoriality of pull backs we have that 
\[
[\BltildeMX^{\lambda}]^{\vv}=i^!_{K^{\lambda}}[\BltildeMXP^{\lambda}]^{\vv}. 
\]
It thus suffices to show that $\tilde{q}^{\lambda}_*i^!_{K^{\lambda}}[\BltildeMXP]^{\vv}=0$. Since pull-backs commute with push forwards we have 
\[
(\tilde{q}^{\lambda})_*i^!_{K^{\lambda}}[\BltildeMXP^\lambda]^{\vv}=i^!_{K^{\lambda}}\tilde{r}^{\lambda}_*[\BltildeMXP^\lambda]^{\vv}
\]
and, since $\lambda \in J$, $\tilde{r}^{\lambda}_*[\BltildeMXP^\lambda]^{\vv}=0$ for dimensional reasons. 
\end{proof}

Propositions~\ref{one} and~\ref{lambda} together prove:

\begin{theorem} \label{virtpushfor}  
  Let $\lambda \in J$.  The morphism $q^{\lambda}$ satisfies the virtual push forward property. 
\end{theorem}

\subsection{A local calculation for CY threefolds}

In this section we consider $X$ a smooth projective Calabi--Yau threefold and we set $n=0$.

\begin{lemma}
Let $\tildeMX^{1,\circ}$ denote the complement of the locus $\tildeMX^0\cap \tildeMX^1$ in $\tildeMX^1$. Then, $\tildeMX^{1,\circ}$ has a virtual class and the morphism $q^1$ is proper on $\tildeMX^{1,\circ}$.
\end{lemma}
\begin{proof} $\tildeMX^{1,\circ}$ is open in $\tildeMX$ and therefore it has a virtual class 
\[
[\tildeMX^{1,\circ}]^{\vv}=[\tildeMX]^{\vv}\cap\tildeMX^{1,\circ}.
\] 
The second statement is clear. 
\end{proof}

\begin{lemma}\label{nicelocus}
  We have 
  \[
    (q^1)_*[\tildeMX^{1,\circ}]^{\vv}=\frac{2+K_X\cdot\beta}{24}[M_0(X)\cap q^1(\tildeMX^{1,\circ}]^{\vv}.
  \]
\end{lemma}
\begin{proof}
  Writing $\xi_1$ for the divisor on $\tildeMX$ which corresponds to smoothing the node $Q$, we see that 
  \[
    C_{\tildeMX/C_{\tildeMP/\tildeV}}\cap \tildeMX^{1,\circ}\simeq C_{\tildeMX^{1,\circ}/\tildeMP^{1,\circ}}\oplus O(\xi_1).
  \]

  At each point $(C,f)\in \tildeMX^{1,\circ}$ we have an exact sequence of vector bundles
  \[
    0\to T_C|_Q\to f^*T_X|_Q\to N_{C/X}|_Q\to 0
  \]
  where $Q$ is the node connecting the contracted genus one component to a rational curve. Let $\mathbb{E}$ be the Hodge bundle on $\bar{M}_{1,1}$. Then, $O(\xi_1)\to f^*T_X|_Q\otimes \E^{\vee}$ is an embedding of vector bundles and
\[
[\tildeMX^{1,\circ}]^{\vv}=c_2(E) \cdot [M_{1,1}(X,0)]\times_X [M_{0,1}(X,d)]^{\vv}
\] 
where $E$ is the cokernel of $O(\xi_1)\to f^*T_X|_Q\otimes \E^{\vee}$.

Theorem~\ref{virtpushfor} implies that 
\[
(q^1)_*\, [\tildeMX^{1,\circ}]^{\vv}=k \, [M_0(X)\cap q^1(\tildeMX^{1,\circ})]^{\vv}
\]
for some $k\in\Q$. We now compute $k$. By commutativity of Chern classes with restrictions we have that $k=c_2(E)\cdot [F]$, for $F$ any fibre of $q_1$. If $F$ denotes the generic fibre of $q^1$, then $F\simeq\pp^1\times \bar{M}_{1,1}$ and $E|_F = N_{\pp^1/X}\oplus \mathbb{E}^{\vee}$. Since
\[
c_2(E|_F)\cdot [F]=\frac{2+K_X\cdot\beta}{24}
\]
the result follows.
\end{proof}

\begin{remark}\label{useless} We have seen in the proof of Lemma \ref{nicelocus} that $q^{\lambda}$ restricted to $\tildeMX^{1,\circ}$ has a perfect dual obstruction theory. Even more, we have a map of relative obstruction theories 
\[ 
E_{\tildeMX/\tildefrPic}\to E_{\tilde{M}_{0,0}(X)/\tildefrPic_{0,0}}.
\] 
One could hope that we have an induced morphism of cones $$C_{\tildeMX/\tildefrPic}\to C_{\tilde{M}_{0,0}(X)/\tildefrPic_{0,0}},$$ which would give a proof of Theorem \ref{virtpushfor}. We do not know if such a morphism exists, because we do not have sufficiently explicit equations for $\tildeMX$ inside $\tildeMP$.
\end{remark}

\begin{theorem}\label{degree one} In notation as above we have 
\[
(q^1)_* \, [M_{1,0}(X,\beta)^1]^{\vv}=\frac{2+K_X\cdot\beta}{24}\,[M_{0,0}(X, \beta)]^{\vv}
\]
and $(q^{\lambda})_*[\tildeMX^{\lambda}]^{\vv}=0$ for $\lambda\neq 1$.
\end{theorem}
\begin{proof} 
By Proposition \ref{lambda}, we only need to compute $(q^1)_*[\tildeMXP]^{\vv}$. (Note that $I=J$ here.)  By Proposition~\ref{one}, we have
\[
(q^1)_*[\tildeMXP]^{\vv}=k [M_{0,1}^X(\pp)],
\] 
for some $k\in\Q$. Since the intersection of $\tildeMXP$ with $M(\pp)^0$ has codimension~$r$ in $\tildeMXP$ and the virtual class of $\tildeMXP$ has codimension~2 in $\tildeMXP$, there are no components of $[\tildeMXP]^{\vv}$ supported on $M(\pp)^0\cap \tildeMXP$. This shows that we can compute $k$ on points in $M_0(X)\cap q^1(\tildeMX^{1,\circ})$ in the following way. By possibly replacing $[M_0(X)]^{\vv}$ with a reduced cycle, we can choose a smooth point $j:pt\in [M_{0,0}^X(\pp)]^{\vv}$ and let $F$ be the fiber of $q^1$ over $pt$. Let us look at the following Cartesian diagram
\[\xymatrix{F\ar[r]\ar[d]_p& M(\pp)^1\ar[d]\\
pt\ar[r]&M_{0,0}(\pp)
}\]
By Remark~\ref{useless} the right vertical arrow has a perfect dual obstruction theory; this gives an induced obstruction theory on $F$ and hence a virtual class which we denote by $[F]^{\vv}$. By commutativity of virtual pull-backs with virtual push-forwards we have that $p_*[F]^{\vv}=k$. By the choice of the point $pt$ , $F$ is a closed substack of $M(\pp)^{1,\circ}$.
Applying Lemma~\ref{nicelocus}, we have that $k=\frac{2+K_X\cdot\beta}{24}$ and thus 
\[
r_*[M^X(\pp)]^{\vv}=\frac{2+K_X\cdot\beta}{24}[M_{0,1}^X(\pp)]^{\vv}.
\]
\end{proof}
 \begin{theorem} Let $X$ be a projective Calabi--Yau threefold. Then the reduced invariants and GW invariants of $X$ are related by the formula 
   \[
     \GW_{1,\beta}^X=\GW^{X,\,\red}_{1,\beta}+\frac{1}{12}\GW^{X}_{0,\beta}.
   \]
 \end{theorem}
 \begin{proof} Combine Corollary~\ref{splitting corollary}, Theorem~\ref{virtpushfor}, and Theorem~\ref{degree one}.
 \end{proof}

\section*{Acknowledgements} 

We thank Francesca Carocci, Huai-Liang Chang and Barbara Fantechi for helpful discussions. We thank Aleksey Zinger for many useful comments on the manuscript.  We are especially grateful to Luca Battistella, both for discussions and for pointing out a serious mistake in an earlier draft.  

T.C.~is supported by ERC Consolidator Grant 682603. C.M.~is supported by an EPSRC-funded Royal Society Dorothy Hodgkin Fellowship. Much of this paper was written while C.M.~was in residence at the Mathematical Sciences Research Institute in Berkeley, California during the Spring 2018 semester,  supported by the National Science Foundation under Grant No. DMS-1440140.

\bibliographystyle{plain}
\bibliography{gen1draft_bibliography}

\end{document}